\documentclass[a4paper,11pt]{amsart}
\usepackage{xcolor}
\usepackage{tikz-cd}
\numberwithin{equation}{section}

\newtheorem{thm}{Theorem}[section]

\newtheorem{lem}[thm]{Lemma}
\newtheorem{cor}[thm]{Corollary}

\theoremstyle{definition}
\newtheorem{defn}[thm]{Definition}

\theoremstyle{remark}
\newtheorem{rem}[thm]{Remark}

\newcommand{\aff}{\mathrm{aff}}
\newcommand{\Z}{\mathbb{Z}}
\newcommand{\R}{\mathbb{R}}
\DeclareMathOperator{\Stab}{Stab}
\DeclareMathOperator{\cInd}{c-Ind}
\newcommand{\trivrep}{\mathbf{1}}
\DeclareMathOperator{\Hom}{Hom}
\DeclareMathOperator{\supp}{supp}
\DeclareMathOperator{\Ind}{Ind}
\DeclareMathOperator{\End}{End}

\DeclareMathOperator{\Ker}{Ker}
\DeclareMathOperator{\val}{val}

\title[Comparison of Hecke modules and mod $p$ representations]{A comparison between pro-$p$-Iwahori Hecke modules and mod $p$ representations}
\author{Noriyuki Abe}
\address{Graduate School of Mathematical Sciences, the University of Tokyo, 3-8-1 Komaba, Meguro-ku, Tokyo 153-8914, Japan.}
\email{abenori@ms.u-tokyo.ac.jp}
\begin{document}
\begin{abstract}
We give an equivalence of categories between certain subcategories of modules of pro-$p$ Iwahori-Hecke algebras and modulo $p$ representations.
\end{abstract}
\maketitle
\section{Introduction}
Let $G$ be a connected reductive $p$-adic group and $K$ a compact open subgroup of $G$.
Then one can attache the Hecke algebra $\mathcal{H}$ to this pair $(G,K)$ and we have a functor $\pi\mapsto \pi^K = \{v\in \pi\mid \pi(k)v = v\ (k\in K)\}$ from the category of smooth representations of $G$ to the category of $\mathcal{H}$-modules.
These algebra and functor are powerful tools to study representation theory of $G$.
In a classical case, namely for smooth representations over the field of complex numbers, this functor gives a bijection between the set of isomorphism classes of irreducible smooth representations of $G$ such that $\pi^K\ne 0$ and the set of isomorphism classes of simple $\mathcal{H}$-modules.
Moreover, the famous theorem of Borel~\cite{MR0444849} says that the functor gives an equivalence of categories between the category of smooth representations $\pi$ of $G$ which is generated by $\pi^K$ and the category of $\mathcal{H}$-modules when $K$ is an Iwahori subgroup.

In this paper, we study modulo $p$ representation theory of $G$.
In this case, it is natural to consider a pro-$p$-Iwahori subgroup $I(1)$ which is the pro-$p$ radical of an Iwahori subgroup since any non-zero modulo $p$ representation has a non-zero vector fixed by the pro-$p$-Iwahori subgroup.
The corresponding Hecke algebra is called a pro-$p$ Iwahori-Hecke algebra.
The aim of this paper is to give a relation between $\mathcal{H}$-modules and modulo $p$ representations.

Such a relation was first discovered by Vign\'eras when $G = \mathrm{GL}_2(\mathbb{Q}_p)$ \cite{MR2392364}.
Based on a classification result due to Barthel-Livn\'e~\cite{MR1361556,MR1290194} and Breuil~\cite{MR2018825}, she proved that the functor $\pi\mapsto\pi^{I(1)}$ gives a bijection between simple objects.
This was enhanced to the level of categories by Ollivier~\cite{MR2572257}.
Namely she proved that the category of $\mathcal{H}$-modules is equivalent to the category of modulo $p$ representations of $G$ which are generated by $\pi^{I(1)}$.
The quasi-inverse of this equivalence is given by $M\mapsto M\otimes_{\mathcal{H}}\cInd_{I(1)}^G\trivrep$ where $\cInd_{I(1)}^G\trivrep$ is the compact induction from the trivial representation of $I(1)$.

However, Ollivier also showed that we cannot expect such correspondence in general.
When $G = \mathrm{GL}_2(F)$ where $F$ is a $p$-adic field such that the number of the residue field is greater than $p$, for a `supersingular' simple module $M$ (we do not recall the definition of supersingular modules since we do not use it in this paper), Ollivier showed that $(M\otimes_{\mathcal{H}}\cInd_{I(1)}^G\trivrep)^{I(1)}$ is not finite-dimensional.
Since simple modules of $\mathcal{H}$ are finite-dimensional, it says that we have no equivalence of categories in this case.

Still we can expect that there is such a correspondence if we `avoid' supersingular representations/modules.
It is proved by Olliver-Schneider \cite{arXiv:1602.00738} that this expectation is true when $G = \mathrm{SL}_2(F)$.
The aim of this paper is to extend this for any $G$.

Let $G$ be a (general) connected reductive $p$-adic group.
In this case, as a consequence of classification theorems~\cite{MR3600042,arXiv:1406.1003_accepted} and the calculation of the invariant part of irreducible representations~\cite{arXiv:1703.10384}, the functor $\pi\mapsto \pi^{I(1)}$ gives a bijection between irreducible modulo $p$ representations of $G$ and simple $\mathcal{H}$-modules which are `far from supersingular representations/modules'.
The aim of this paper is to generalize this correspondence to the level of categories.
More precisely, we prove the equivalence of the following two categories.
\begin{itemize}
\item The category of $\mathcal{H}$-modules $M$ such that $\dim(M)<\infty$ and a certain element of the center of $\mathcal{H}$ is invertible on $M$ (see Definition~\ref{defn:C}).
\item The category of modulo $p$ representations $\pi$ of $G$ such that
\begin{itemize}
\item $\pi$ is generated by $\pi^{I(1)}$.
\item $\pi$ has a finite length.
\item any irreducible subquotient of $\pi$ is isomorphic to a subquotient of $\Ind_B^G\sigma$ where $B$ is a minimal parabolic subgroup and $\sigma$ is an irreducible representation of the Levi quotient of $B$.
\end{itemize}
\end{itemize}
Note that, an $\mathcal{H}$-module $M$ is supersingular if and only if certain elements in the center of $\mathcal{H}$ act by zero and a modulo $p$ irreducible admissible representation $\pi$ of $G$ is supersingular if and only if it is supercuspidal, namely it does not appear as a subquotient of a parabolically induced representation from an irreducible admissible representation of a proper Levi subgroup.
Therefore some conditions as above says that $M$ (resp.\ $\pi$) is `far from supersingular modules (resp.\ representations)'.

We give an outline of the proof.
Since the correspondence is true for irreducible representations, by induction on the length, it is sufficient to prove the following (Theorem~\ref{thm:injective}).
Let $M$ be an $\mathcal{H}$-module which we are considering.
Then $M\to M\otimes_{\mathcal{H}}\Ind_{I(1)}^G\trivrep$ is injective.
This theorem is proved in Section~\ref{sec:C and injectivity}.
Here are some reductions.
\begin{itemize}
\item Let $\mathcal{A}$ be the Bernstein subalgebra introduced in \cite{MR3484112}.
Since we have an embedding $M\hookrightarrow \Hom_{\mathcal{A}}(\mathcal{H},M)$, it is sufficient to prove the theorem for $\Hom_{\mathcal{A}}(\mathcal{H},M)$.
\item We have a decomposition of $M|_{\mathcal{A}}$ along the `support' (Definition~\ref{defn:support}). We may assume that the support of $M$ is contained in a Weyl chamber.
\item Using a result in \cite{arXiv:1406.1003_accepted}, parabolic inductions and a result of Ollivier-Vign\'eras~\cite{arXiv:1703.04921}, we may assume that the support is the dominant Weyl chamber.
\item We can rewrite $\Hom_{\mathcal{A}}(\mathcal{H},M)$ as $M'\otimes_{\mathcal{A}}\mathcal{H}$ for some $M'$.
Hence it is sufficient to prove that $M'\otimes_{\mathcal{A}}\mathcal{H}\to M'\otimes_{\mathcal{A}}\cInd_{I(1)}^G\trivrep$ is injective.
\end{itemize}
By a result in \cite{MR3666048}, both $M\otimes_{\mathcal{A}}\mathcal{H}$ and $M\otimes_{\mathcal{A}}\cInd_{I(1)}^G\trivrep$ relate $\cInd_K^GV$ where $K$ is a special parahoric subgroup and $V$ a certain representation of $K$.
The structure of this representations is studied in \cite{MR3600042} and using such result we prove the injectivity.

It is almost immediate to prove our main theorem from the above injectivity.
This is done in Section~\ref{sec:Theorem}.

\subsection*{Acknowledgment}
The author was supported by JSPS KAKENHI Grant Number 18H01107.

\section{Notation and Preliminaries}\label{sec:Preliminaries}
Let $F$ be a non-archimedean local field of residue characteristic $p$ and $G$ a connected reductive group over $F$.
Let $C$ be an algebraically closed field of characteristic $p$.
This is the coefficient field of representations in this paper.
\emph{All representations in this paper are smooth representations over $C$.}

In general, for any algebraic group $H$ over $F$, we denote the group of valued points $H(F)$ by the same letter $H$.
Fix a maximal split torus $S$ of $G$ and minimal parabolic subgroup $B$ containing $S$.
The centralizer $Z$ of $S$ in $G$ is a Levi subgroup of $B$.
We denote the unipotent radical of $B$ by $U$ and the opposite of $B$ containing $Z$ by $\overline{B}$.
The unipotent radical of $\overline{B}$ is denoted by $\overline{U}$.

Consider the reduced apartment corresponding to $S$ and take an alcove $\mathbf{A}_0$ and a special point $\mathbf{x}_0$ from the closure of $\mathbf{A}_0$.
Let $K$ be the special parahoric subgroup corresponding to $\mathbf{x}_0$ and $I$ the Iwahori subgroup determined by $\mathbf{A}_0$.
Let $I(1)$ be the pro-$p$ Iwahori subgroup attached to $\mathbf{A}_0$, namely the pro-$p$ radical of $I$.
The space of $C$-valued compactly supported $I(1)$-biinvariant functions $\mathcal{H}$ has a structure of a $C$-algebra via the convolution product.
The algebra $\mathcal{H}$ is called pro-$p$ Iwahori-Hecke algebra.
The structure of this algebra is studied by Vign\'eras~\cite{MR3484112}.

Let $N_G(S)$ be the normalizer of $S$ in $G$ and put $W_0 = N_G(S)/Z$, $W = N_G(S)/(Z\cap K)$ and $W(1) = N_G(S)/(Z\cap I(1))$.
Let $G'$ be the subgroup of $G$ generated by $U$ and $\overline{U}$.
Note that this is not a group of the valued points of an algebraic group in general.
Let $W_\aff$ be the image of $G'\cap N_G(S)$ in $W$.
The action of $W_\aff$ on the apartment is faithful and therefore it is a subgroup of the group of affine transformations of the apartment.
Let $S_\aff$ be the set of reflections along the walls of $\mathbf{A}_0$.
Then $(W_\aff,S_\aff)$ is a Coxeter system.
Denote its length function by $\ell$.
Let $N_{W}(\mathbf{A}_0)$ be the stabilizer of $\mathbf{A}_0$ in $W$.
Then the group $W$ is the semi-direct product of $W_\aff$ and $N_W(\mathbf{A}_0)$.
The function $\ell$ is extended to $W$, trivially on $N_W(\mathbf{A}_0)$.
We also inflate $\ell$ to $W(1)$ via $W(1)\to W$.
We have the Bruhat order on $(W_\aff,S_\aff)$ and we extend it to $W$ by: $w_1\omega_1 < w_2\omega_2$ if and only if $w_1 < w_2$ and $\omega_1 = \omega_2$ where $w_1,w_2\in W_\aff$ and $\omega_1,\omega_2\in N_W(\mathbf{A}_0)$.
For $w_1,w_2\in W(1)$, we say $w_1 < w_2$ if $\overline{w}_1 < \overline{w}_2$ where $\overline{w}_i$ is the image of $w_i$ in $W$ ($i = 1,2$).
As usual we say $w_1 \le w_2$ if and only if $w_1 < w_2$ or $w_1 = w_2$.

We give some of structure theorems of $\mathcal{H}$.
For $w\in W(1)$, let $T_w$ be the characteristic function on $I(1)\widetilde{w}I(1)$ where $\widetilde{w}\in N_G(S)$ is a lift of $w$.
Then $T_w$ does not depend on the choice of a lift and, since we have the bijection $I(1)\backslash G/I(1)\simeq W(1)$, $\{T_w\mid w\in W(1)\}$ is a basis of $\mathcal{H}$.
This basis is called Iwahori-Matsumoto basis.
This basis satisfies the following braid relations:
\[
\text{$T_{w_1}T_{w_2} = T_{w_1w_2}$ if $\ell(w_1w_2) = \ell(w_1) + \ell(w_2)$}
\]
where $w_1,w_2\in W(1)$.
Let $Z_\kappa = (Z\cap K)/(Z\cap I(1))$.
Then this is a subgroup of $W(1)$.
Since any elements in $Z_\kappa$ has the length $0$ (since it is in the kernel of $W(1)\to W$), from the braid relations, we have $T_{t_1}T_{t_2} = T_{t_1t_2}$ for $t_1,t_2\in Z_\kappa$.
In other words, the embedding $C[Z_\kappa]\hookrightarrow \mathcal{H}$ defined by $\sum_{t\in Z_\kappa}c_{t}t\mapsto \sum_{t\in Z_\kappa}c_{t}T_{t}$ is an algebra homomorphism where $C[Z_\kappa]$ is the group ring of $Z_\kappa$.
Using this embedding, we regard $C[Z_\kappa]$ as a subalgebra of $\mathcal{H}$.

Let $S_\aff(1)$ be the inverse image of $S_\aff$ in $W(1)$.
Then for $s\in S_\aff(1)$, we have
\[
T_s^2 = c_sT_s
\]
for some $c_s\in C[Z_\kappa]$.
An element $c_s$ is given in \cite[4.2]{MR3484112}.

Define $T_w^*$ as in \cite[4.3]{MR3484112} for $w\in W(1)$.
This is also a basis of $\mathcal{H}$ and it satisfies the following: $T_w^* \in T_w + \sum_{v < w}CT_v$ and $T_{w_1}^*T_{w_2}^* = T_{w_1w_2}^*$ if $\ell(w_1w_2) = \ell(w_1) + \ell(w_2)$.

Let $o$ be a spherical orientation~\cite[5.2]{MR3484112}.
Note that the set of spherical orientations are canonially bijective with the set of Weyl chambers.
For each $o$, we have another basis $\{E_o(w)\mid w\in W(1)\}$ defined in \cite[5.3]{MR3484112}.
The orientations correspond to the Weyl chambers.
Let $o_-$ be the orientation corresponding to the anti-dominant Weyl chamber and set $E(w) = E_{o_-}(w)$.

Set $\Lambda(1) = Z/(Z\cap I(1))$.
This is a subgroup of $W(1)$.
For $\lambda_1,\lambda_2\in \Lambda(1)$, the multiplication $E(\lambda_1)E(\lambda_2)$ is simply given.
To say it, we give some notation.
The pair $(G,S)$ gives a root datum $(X^*(S),\Sigma,X_*(S),\Sigma^\vee)$ and since we have fixed a Borel subgroup we also have a positive system $\Sigma^+\subset \Sigma$ and the set of simple roots $\Delta\subset\Sigma^+$.
An element $v\in X_*(S)\otimes_{\Z}\R$ is called dominant if and only if $\langle v,\alpha\rangle\ge 0$ for any $\alpha\in\Sigma^+$.
A $W_0$-orbit of the set of dominant elements is called a closed Weyl chamber.
We also say that $v\in X_*(S)\otimes_{\Z}\R$ is regular if $\langle v,\alpha\rangle\ne 0$ for any $\alpha\in\Sigma$.
We have a homomorphism $\nu\colon Z\to X_*(S)\otimes_{\Z}\R = \Hom_\Z(X^*(S),\R)$ characterized by $\nu(z)(\chi) = -\val(\chi(z))$ where $z\in S$, $\chi\in X^*(S)$ and $\val\colon F^\times\to \Z$ is the normalized valuation.
This homomorphism factors through $Z\to \Lambda(1)$ and the induced homomorphism $\Lambda(1)\to X_*(S)\otimes_{\Z}\R$ is denoted by the same letter $\nu$.
We let $\Lambda^+(1)$ the set of $\lambda\in \Lambda(1)$ such that $\nu(\lambda)$ is dominant.
For $w\in W_0$, let $w(\Lambda^+(1))$ be the set of $\lambda\in \Lambda(1)$ such that $w^{-1}(\nu(\lambda))$ is dominant.

The multiplication $E(\lambda_1)E(\lambda_2)$ is $E(\lambda_1\lambda_2)$ if $\nu(\lambda_1)$ and $\nu(\lambda_2)$ are in the same closed Weyl chamber (in other words, $\lambda_1,\lambda_2\in w(\Lambda^+(1))$ for some $w\in W_0$) and otherwise it is zero.
In particular, $\mathcal{A} = \bigoplus_{\lambda\in \Lambda(1)}CE(\lambda)$ is a subalgebra of $\mathcal{H}$.
If we fix a closed Weyl chamber $\mathcal{C}$, then $\bigoplus_{\nu(\lambda)\in \mathcal{C}}CE(\lambda)$ is a subalgebra of $\mathcal{A}$ and the linear map
\[
\bigoplus_{\nu(\lambda)\in \mathcal{C}}CE(\lambda)\to C[\Lambda(1)]
\]
defined by $E(\lambda)\mapsto \tau_\lambda$ is an algebra embedding.
Here $C[\Lambda(1)]$ is the group ring of $\Lambda(1)$ and we denote the element in $C[\Lambda(1)]$ corresponding to $\lambda\in \Lambda(1)$ by $\tau_\lambda$, namely $C[\Lambda(1)] = \bigoplus_{\lambda\in \Lambda(1)}C\tau_\lambda$.
\begin{rem}\label{rem:inverse of E(lambda)}
\begin{enumerate}
\item If $\langle\nu(\lambda),\alpha\rangle = 0$ for any $\alpha\in\Sigma$, then $\nu(\lambda)$ and $\nu(\lambda^{-1})$ are in the same closed Weyl chamber. (In fact, $\nu(\lambda)$ and $\nu(\lambda^{-1})$ are in any closed Weyl chamber.) Hence $E(\lambda)E(\lambda^{-1}) = 1$.
In particular, $E(\lambda)$ is invertible.
\item If $\lambda\in \Lambda(1)$ is in the center of $\Lambda(1)$, then $E(\lambda)$ is also in the center of $\mathcal{A}$.
This follows from the above description of the multiplication.
\end{enumerate}
\end{rem}

Let $J$ be a subset of $\Delta$ and denote the corresponding standard parabolic subgroup by $P_J$.
Let $L_J$ be the Levi part of $P_J$ containing $Z$.
Then $K\cap L_J$ is a special parahoric subgroup and $I(1)_J = I(1)\cap L_J$ a pro-$p$ Iwahori subgroup.
Attached to these, we have many objects.
For such objects we add a suffix $J$, for example, the pro-$p$ Iwahori-Hecke algebra attached to $(L_J,I(1)_J)$ is denoted by $\mathcal{H}_J$.
There are two exceptions: base $T_w$ and $E(w)$ for $\mathcal{H}_J$ is denoted by $T^J_w$ and $E^J(w)$, respectively.
For each $J\subset \Delta$, we have two subalgebras $\mathcal{H}_J^+$, $\mathcal{H}_J^-$ of $\mathcal{H}_J$ and four algebra homomorphisms $j_J^+,j_J^{+*}\colon \mathcal{H}_J^+\to \mathcal{H}$ and $j_J^-,j_J^{-*}\colon \mathcal{H}_J^-\to \mathcal{H}$.
See \cite[2.8]{arXiv:1612.01312} for the definitions. (Here $\mathcal{H}_J^+$ is denoted by $\mathcal{H}_{P_J}^+$ in \cite{arXiv:1612.01312}, etc.)

\section{The category $\mathcal{C}$ and a proof of the injectivity}\label{sec:C and injectivity}
\subsection{The category $\mathcal{C}$}
The modules in this paper are right modules unless otherwise stated.
In this paper, we focus on the full-sub category $\mathcal{C}$ of the category of $\mathcal{H}$-modules defined using the center $\mathcal{Z}$ of $\mathcal{H}$.
The center $\mathcal{Z}$ is described using the basis $\{E(w)\}$.
Since $\Lambda(1)$ is normal in $W(1)$, the group $W(1)$ acts on $\Lambda(1)$ by the conjugate action.
For $\lambda\in \Lambda(1)$ denote the orbit through $\lambda$ by $\mathcal{O}_\lambda$.
For $\lambda \in \Lambda(1)$, put $z_\lambda = \sum_{\lambda'\in \mathcal{O}_\lambda}E(\lambda')$.
Then $\{z_\lambda\mid z\in \Lambda(1)/W(1)\}$ gives a basis of $\mathcal{Z}$~\cite[Theorem~1.2]{MR3271250}.
Fix a uniformizer $\varpi$ of $F$ and let $\Lambda_S(1)$ be the image of $\{\xi(\varpi)\mid \xi\in X_*(S)\}$.
\begin{defn}\label{defn:C}
An $\mathcal{H}$-module $M$ is in $\mathcal{C}$ if and only if $z_\lambda$ is invertible on $M$ for any $\lambda\in \Lambda_S(1)$.
\end{defn}
\begin{lem}\label{lem:description of z_lambda}
Let $\lambda\in \Lambda_S(1)$.
Then we have the following.
\begin{enumerate}
\item For $w\in W(1)$, $w$ stabilizes $\lambda$ if and only if the image of $w$ in $W_0$ stabilizes $\nu(\lambda)$.
\item Let $\{w_1,\dots,w_r\}\subset W(1)$ be a subset of $W(1)$ such that the image in $W_0$ gives a set of complete representatives of $W_0/\Stab_{W_0}(\nu(\lambda))$.
Then we have $z_\lambda = \sum_{i = 1}^rE(w_i \lambda w_i^{-1})$. (Note that $w_i\lambda w_i^{-1}$ depends only on the image of $w_i$ in $W_0/\Stab_{W_0}(\lambda)$ by (1).)
\end{enumerate}
\end{lem}
\begin{proof}
Take $\xi\in X_*(S)$ such that $\lambda = \xi(\varpi)^{-1}$.
We have $\nu(\lambda) = \xi$.
Let $w\in W(1)$ and denote the image of $w$ in $W_0$ by $w_0$.
Then we have $w\lambda w^{-1} = (w_0\xi)(a)^{-1}$.
Hence if $w_0$ stabilizes $\xi = \nu(\lambda)$, then $w$ stabilizes $\lambda$.
Obviously if $w$ stabilizes $\lambda$ then $w_0$ stabilizes $\nu(\lambda)$.

By (1), $\Stab_{W(1)}(\lambda)$ is the inverse image of $\Stab_{W_0}(\lambda)$. Therefore we have $W(1)/\Stab_{W(1)}(\lambda)\simeq W_0/\Stab_{W_0}(\lambda)$.
By the definition, we have $z_\lambda = \sum_{w\in W(1)/\Stab_{W(1)}(\lambda)}E(w\lambda w^{-1})$.
Hence we get (2).

\end{proof}\begin{lem}\label{lem:multiplication in the center (easy)}
Let $\lambda,\mu\in \Lambda_S(1)$ and assume that $\nu(\lambda)$ and $\nu(\mu)$ are in the same closed Weyl chamber.
We also assume that $\nu(\lambda)$ is regular.
Then we have $z_\lambda z_\mu = z_{\lambda \mu}$.
\end{lem}
\begin{proof}
Take $w_1,\dots,w_r\in W(1)$ such that the images of them in $W_0$ gives a set of complete representatives of $W_0/\Stab_{W_0}(\nu(\mu))$.
Then we have $z_\mu = \sum_i E(w_i\mu w_i^{-1})$ by the above lemma.
Let $v_1,\dots,v_s$ be a set of complete representatives of $W_0 = W(1)/\Lambda(1)$.
Then we have $z_\lambda = \sum_j E(v_i\lambda v_i^{-1})$. (Note that $\nu(\lambda)$ is assumed to be regular.)
Since $\nu(\lambda)$ is regular, for each $i$, there exists only one $j_i = 1,\dots,r$ such that $v_i(\nu(\lambda))$ and $w_{j_i}(\nu(\mu))$ is in the same closed Weyl chamber.
Hence we get
\[
E(v_i\lambda v_i^{-1})E(w_j\mu w_j^{-1}) = 
\begin{cases}
0 & (j\ne j_i),\\
E(v_i\lambda v_i^{-1}w_j\mu w_j^{-1}) & (j = j_i).
\end{cases}
\]
Moreover, $\nu(\lambda)$ and $v_i^{-1}w_{j_i}(\nu(\mu))$ is in the same closed Weyl chamber.
Since $\nu(\lambda)$ and $\nu(\mu)$ are in the same closed Weyl chamber by the assumption, we get $v_i^{-1}w_{j_i}(\nu(\mu)) = \nu(\mu)$.
Therefore $v_i^{-1}w_{j_i}$ stabilizes $\nu(\mu)$.
As in the previous lemma, $v_i^{-1}w_{j_i}$ also stabilizes $\mu$.
Hence $w_{j_i}\mu w_{j_i}^{-1} = v_i\mu v_i^{-1}$.
We get
\[
E(v_i\lambda v_i^{-1})E(w_j\mu w_j^{-1}) = 
\begin{cases}
0 & (j\ne j_i),\\
E(v_i\lambda\mu v_i^{-1}) & (j = j_i).
\end{cases}
\]
Now we get
\[
z_\lambda z_\mu = \sum_i\sum_j E(v_i\lambda v_i^{-1})E(w_j\mu w_j^{-1}) = \sum_i E(v_i\lambda \mu v_i^{-1}).
\]
By the assumption, $\nu(\lambda\mu)$ is regular and $\lambda\mu\in\Lambda_S(1)$.
Hence the last term is $z_{\lambda\mu}$ by the above lemma.
\end{proof}
\begin{lem}
An $\mathcal{H}$-module $M$ is in $\mathcal{C}$ if and only if for some $\lambda\in \Lambda_S(1)$ such that $\nu(\lambda)$ is regular, the element $z_\lambda$ is invertible on $M$.
\end{lem}
\begin{proof}
Assume that there exists $\lambda_0\in \Lambda_S(1)$ such that $\nu(\lambda_0)$ is regular and $z_{\lambda_0}$ is invertible on $M$.
Let $\lambda\in \Lambda_S(1)$ and we prove that $\lambda$ is also invertible on $M$.
Replacing $\lambda$ with an element in the orbit through $\lambda$, we may assume that $\nu(\lambda)$ and $\nu(\lambda_0)$ are in the same closed Weyl chamber.
Take a sufficiently large $n\in \Z_{>0}$ such that $\nu(\lambda_0^n\lambda^{-1})$ is also in the same closed Weyl chamber as $\nu(\lambda_0)$.
Set $\mu = \lambda_0^n\lambda^{-1}$.
Then by the above lemma, we have $z_\mu z_\lambda = z_{\lambda_0^n} = z_{\lambda_0}^n$.
By the assumption, $z_{\lambda_0}^n$ is invertible on $M$.
Hence $z_\lambda$ is invertible, namely we have $M\in \mathcal{C}$.
\end{proof}

\subsection{Theorem}
In the rest of this section, we prove the following 
\begin{thm}\label{thm:injective}
If $M\in \mathcal{C}$, then $M\to M\otimes_{\mathcal{H}}\cInd_{I(1)}^G\trivrep$ is injective.
\end{thm}

\subsection{Reductions}
Define a subalgebra $\mathcal{A}$ of $\mathcal{H}$ by $\mathcal{A} = \bigoplus_{\lambda\in \Lambda(1)}CE(\lambda)$.
Let $M\in \mathcal{C}$ and set $M' = \Hom_\mathcal{A}(\mathcal{H},M)$.
Defining the action of $X\in \mathcal{H}$ on $M'$ by $(\varphi X)(Y) = \varphi(XY)$ for $\varphi\in\Hom_{\mathcal{A}}(\mathcal{H},M)$ and $Y\in \mathcal{H}$, $M'$ is a right $\mathcal{H}$-module.
The map $m\mapsto (X\mapsto mX)$ gives an $\mathcal{H}$-module embedding $M\hookrightarrow M'$ and we have the following commutative diagram:
\[
\begin{tikzcd}
M \arrow[d,hookrightarrow]\arrow[r] & M\otimes_{\mathcal{H}}\cInd_{I(1)}^G\trivrep \arrow[d]\\
M'\arrow[r] & M'\otimes_{\mathcal{H}}\cInd_{I(1)}^G\trivrep.
\end{tikzcd}
\]
Therefore, to prove Theorem~\ref{thm:injective}, it is sufficient to prove that the map $M'\to M'\otimes_{\mathcal{H}}\cInd_{I(1)}^G\trivrep$ is injective.

\begin{lem}\label{lem:decomposition along support}
Any module $M \in \mathcal{C}$ has a functorial decomposition $M = \bigoplus_{w\in W_0}M_w$ as an $\mathcal{A}$-module such that $E(\mu)$ acts on $M_w$ by
\begin{itemize}
\item zero if $w^{-1}\nu(\mu)$ is not dominant.
\item invertible if $w^{-1}\nu(\mu)$ is dominant.
\end{itemize}
\end{lem}
\begin{proof}
Fix $\lambda_0\in \Lambda_S(1)$ such that $\nu(\lambda_0)$ is regular dominant.
Put $\lambda_w = n_w\lambda_0n_w^{-1}$ and set $M_w = ME(\lambda_w)$.
Since $\lambda_w\in \Lambda_S(1)$ is central, $E(\lambda_w)$ is also central in $\mathcal{A}$.
Hence $M_w$ is an $\mathcal{A}$-submodule.

We prove that $\lambda_w$ is invertible on $M_w$.
Since $\nu(\lambda_0)$ is regular, $\nu(\lambda_v)$ and $\nu(\lambda_w)$ are not in the same closed Weyl chamber if $v\ne w$.
Therefore $E(\lambda_v)E(\lambda_w) = 0$.
Hence $M_wE(\lambda_v) = 0$ if $v\ne w$.
Therefore for $m\in M_w$, we have $mz_{\lambda_0} = \sum_{v\in W_0}mE(\lambda_v) = mE(\lambda_w)$.
Hence if $mE(\lambda_w) = 0$ then $mz_{\lambda_0} = 0$, hence $m = 0$ since $z_{\lambda_0}$ is invertible.
Therefore $E(\lambda_w)$ is injective on $M_w$.
We also have that $mz_{\lambda_0}^2 = mE(\lambda_w)z_{\lambda_0} = mz_{\lambda_0}E(\lambda_w) = mE(\lambda_w)^2$ since $z_{\lambda_0}$ commuts with $E(\lambda_w)$. (Recall that $z_{\lambda_0}$ is in the center of $\mathcal{H}$.)
Hence $m = m_0E(\lambda_w)$ where $m_0 = mz_{\lambda_0}^{-2}E(\lambda_w)\in M_w$.
Therefore $E(\lambda_w)$ is surjective on $M_w$.

For $\mu \in \Lambda(1)$ such that $w^{-1}(\nu(\mu))$ is not dominant, $\nu(\mu)$ and $\nu(\lambda_w)$ are not in the same closed Weyl chamber.
Hence $E(\mu)E(\lambda_w) = 0$.
Therefore $E(\mu) = 0$ on $M_w$.
On the other hand, assume that $w^{-1}(\nu(\mu))$ is dominant.
Then $\nu(\mu)$ and $\nu(\lambda_w)$ is in the same closed Weyl chamber.
Take sufficiently large $n\in \Z_{\ge 0}$ such that $\nu(\lambda_w^n\mu^{-1})$ is also in the same closed Weyl chamber as $\nu(\mu)$.
Then we have $E(\lambda_w)^n = E(\lambda_w^n) = E(\lambda_w^n\mu^{-1})E(\mu)$.
Since $E(\lambda_w)$ is invertible on $M_w$, $E(\mu)$ is also invertible on $M_w$.

We prove $M = \bigoplus_{w\in W_0}M_w$.
Since $z_{\lambda_0}$ is invertible, any element in $M$ can be written $mz_{\lambda_0}$ for some $m\in M$.
We have $mz_{\lambda_0} = \sum_{w\in W_0}mE(\lambda_w)\in \sum_{w\in W_0}M_w$.
Hence $M = \sum_{w\in W_0}M_w$.
Let $m_w\in M_w$ and assume that $\sum_{w\in W_0}m_w = 0$.
Then for each $v\in W_0$ we have $\sum_{w\in W_0}m_wE(\lambda_v) = 0$.
Since $m_wE(\lambda_v) = 0$ for $v\ne w$, we have $m_vE(\lambda_v) = 0$.
Since the action of $E(\lambda_v)$ on $M_v$ is invertible, $m_v = 0$.
\end{proof}

Since $\Hom_{\mathcal{A}}(\mathcal{H},M) = \bigoplus_{w\in W_0}\Hom_{\mathcal{A}}(\mathcal{H},M_w)$, to prove $M'\to M'\otimes_{\mathcal{H}}\cInd_{I(1)}^G\trivrep$ is injective, it is sufficient to prove that the homomorphism $\Hom_{\mathcal{A}}(\mathcal{H},M_w)\to \Hom_{\mathcal{A}}(\mathcal{H},M_w)\otimes_{\mathcal{H}}\cInd_{I(1)}^G\trivrep$ is injective.
\begin{defn}\label{defn:support}
We say that $\supp M = w(\Lambda^+(1))$ if and only if $E(\lambda)$ is
\begin{itemize}
\item zero if $w^{-1}(\nu(\mu))$ is not dominant.
\item invertible if $w^{-1}(\nu(\mu))$ is dominant.
\end{itemize}
for any $\lambda\in \Lambda(1)$.
\end{defn}

From the above discussions, to prove Theorem~\ref{thm:injective}, it is sufficient to prove the following lemma.
\begin{lem}\label{lem:injective, supported on a chamber}
Let $M$ be an $\mathcal{A}$-module such that $\supp M = w(\Lambda^+(1))$ where $w\in W_0$.
Then $\Hom_{\mathcal{A}}(\mathcal{H},M)\to \Hom_{\mathcal{A}}(\mathcal{H},M)\otimes_{\mathcal{H}}\cInd_{I(1)}^G\trivrep$ is injective.
\end{lem}

We take a lift $n_w$ of each $w\in W_0$ in $W(1)$ such that $n_{w_1w_2} = n_{w_1}n_{w_2}$ if $\ell(w_1w_2) = \ell(w_1) + \ell(w_2)$.
Let $M$ be an $\mathcal{A}$-module and $w\in W_0$.
We define a new $\mathcal{A}$-module $n_wM$ as follows.
As a vector space, $n_wM = M$ and the action of $E(\lambda)\in \mathcal{A}$ on $n_wM$ is the action of $E(n_w^{-1}\lambda n_w)$ on $M$.
This defines an auto-equivalence of the category of $\mathcal{A}$-modules.
If $\supp M = v(\Lambda^+(1))$, then $\supp n_wM = wv(\Lambda^+(1))$.
With this notation, Lemma~\ref{lem:injective, supported on a chamber} is equivalent to the following.
\begin{lem}\label{lem:injective, twist form}
Let $M$ be an $\mathcal{A}$-module such that $\supp M = \Lambda^+(1)$.
Then the map $\Hom_\mathcal{A}(\mathcal{H},n_wM)\to \Hom_\mathcal{A}(\mathcal{H},n_wM)\otimes_{\mathcal{H}}\cInd_{I(1)}^G\trivrep$ is injective.
\end{lem}

\subsection{Reduction to $w = w_J$ for some $J\subset \Delta$}\label{subsec:Reduction to w=w_J}
For a subset $J\subset \Delta$, let $w_J$ be the longest element in $W_{0,J}$.
We prove that we may assume $w = w_J$ for some $J$ in Lemma~\ref{lem:injective, twist form}.

We relate our $M$ with modules studied in \cite{arXiv:1406.1003_accepted}.
Consider the homomorphism $\mathcal{A}\to C[\Lambda(1)]$ defined by
\begin{equation}\label{eq:hom chi tilde}
E(\lambda)\mapsto \begin{cases}\tau_\lambda & (\lambda\in\Lambda^+(1)),\\0 & (\text{otherwise}).\end{cases}
\end{equation}
We regard $C[\Lambda(1)]$ as a right $\mathcal{A}$-module via this homomorphism.
For $w\in W_0$, we also have the $\mathcal{A}$-module $n_wC[\Lambda(1)]$.
Then we consider the module
\[
n_wC[\Lambda(1)]\otimes_{\mathcal{A}}\mathcal{H}.
\]
This is a $(C[\Lambda(1)],\mathcal{H})$-bimodule.

Let $M$ be an $\mathcal{A}$-module such that $\supp M = \Lambda^+(1)$.
Then we define a structure of $C[\Lambda(1)]$-module on $M$ by
\[
m\tau_{\lambda_1\lambda_2^{-1}} = mE(\lambda_1)E(\lambda_2)^{-1}
\]
where $\lambda_1,\lambda_2\in \Lambda^+(1)$ and $m\in M$.
(Since $\supp M = \Lambda^+(1)$, $E(\lambda_2)$ is invertible on $M$.)
It is easy to see that this definition is well-defined and define a structure of $C[\Lambda(1)]$-module.
Then we have
\[
M\otimes_{C[\Lambda(1)]}n_wC[\Lambda(1)]\simeq n_wM.
\]
The isomorphisms are given by $m\otimes f\mapsto mf$ from the left hand side to the right hand side and $m\mapsto m\otimes 1$ in the opposite direction.
Therefore we have 
\[
n_wM\otimes_{\mathcal{A}}\mathcal{H}\simeq M\otimes_{C[\Lambda(1)]}\otimes n_wC[\Lambda(1)]\otimes_{\mathcal{A}}\mathcal{H}.
\]

For each $w\in W_0$, set $\Delta_w = \{\alpha\in\Delta\mid w(\alpha)>0\}$.
Then by \cite[Theorem~3.13]{arXiv:1406.1003_accepted}, if $\Delta_{w_1} = \Delta_{w_2}$, we have
\[
n_{w_1}C[\Lambda(1)]\otimes_{\mathcal{A}}\mathcal{H}\simeq n_{w_2}C[\Lambda(1)]\otimes_{\mathcal{A}}\mathcal{H}.
\]
Therefore we get (1) of the next lemma.
\begin{lem}
Let $M$ be as in Lemma~\ref{lem:injective, twist form}.
If $w_1,w_2\in W_0$ satisfies $\Delta_{w_1} = \Delta_{w_2}$, then we have
\begin{enumerate}
\item $n_{w_1}M\otimes_{\mathcal{A}}\mathcal{H}\simeq n_{w_2}M\otimes_{\mathcal{A}}\mathcal{H}$.
\item $\Hom_{\mathcal{A}}(\mathcal{H},n_{w_1}M)\simeq \Hom_{\mathcal{A}}(\mathcal{H},n_{w_2}M)$.
\end{enumerate}
\end{lem}
\begin{proof}
We have proved (1).
We prove (2).

Let $\iota$ be an automorphism of $\mathcal{H}$ defined in \cite[Proposition~4.23]{MR3484112} and $\zeta\colon \mathcal{H}\to \mathcal{H}$ an anti-automorphism defined by $\zeta(T_w) = T_{w^{-1}}$.
(The linear map $\zeta$ is an anti-homomorphism by \cite[4.1]{arXiv:1406.1003_accepted}.)
Set $f = \iota\circ\zeta$.
Since $\zeta(E(\lambda)) = E_{o_+}(\lambda^{-1})$ \cite[Lemma~4.3]{arXiv:1406.1003_accepted} and $\iota(E_{o_+}(\lambda)) = (-1)^{\ell(\lambda)}E(\lambda)$ \cite[Lemma~5.31]{MR3484112}, we have $f(E(\lambda)) = (-1)^{\ell(\lambda)}E(\lambda^{-1})$.
In particular, $f$ preserves $\mathcal{A}$.
It is easy to see $f^2(T_w) = T_w$ for any $w\in W(1)$.
Hence $f^2$ is identity.

For a left $\mathcal{H}$-module $N$, we define a right $\mathcal{H}$-module $N^f$ by $N^f = N$ as a vector space and the action of $X\in \mathcal{H}$ on $N^f$ is the action of $f(X)$ on $N$.
Then $m\otimes X\mapsto f(X)\otimes m$ gives an isomorphism $(N^f\otimes_{\mathcal{A}}\mathcal{H})^f\simeq \mathcal{H}\otimes_{\mathcal{A}}N$.

For a right $\mathcal{H}$-module (resp.\ $\mathcal{A}$-module) $L$, set $L^* = \Hom_C(L,C)$.
Then this is a left $\mathcal{H}$-module (resp.\ $\mathcal{A}$-module).
Let $M$ be as in the lemma.
Since $f(E(\lambda)) = (-1)^{\ell(\lambda)}E(\lambda^{-1})$, we have $\supp (n_{w_1}M^*)^f = w_1(\Lambda^+(1)^{-1}) = w_1w_\Delta(\Lambda^+(1))$.
Hence $(n_{w_1}M^*)^f = n_{w_1w_\Delta}M'$ for some $\mathcal{A}$-module $M'$ such that $\supp M' = \Lambda^+(1)$.
Since $\Delta_{w_1w_\Delta} = \Delta\setminus(-w_\Delta(\Delta_{w_1}))$, we also have $\Delta_{w_1w_\Delta} = \Delta_{w_2w_\Delta}$.
Hence by (1), we get $n_{w_1w_\Delta}M'\otimes_{\mathcal{A}}\mathcal{H}\simeq n_{w_2w_\Delta }M'\otimes_{\mathcal{A}}\mathcal{H}$.
Therefore we get $(n_{w_1}M^*)^f\otimes_{\mathcal{A}}\mathcal{H}\simeq (n_{w_2}M^*)^f\otimes_{\mathcal{A}}\mathcal{H}$.
Applying $(\cdot)^f$ to the both sides and using $(N^f\otimes_{\mathcal{A}}\mathcal{H})^f\simeq \mathcal{H}\otimes_{\mathcal{A}}N$, we get $\mathcal{H}\otimes_{\mathcal{A}}n_{w_1}M^*\simeq \mathcal{H}\otimes_{\mathcal{A}}n_{w_2}M^*$.
Hence we have $(\mathcal{H}\otimes_{\mathcal{A}}n_{w_1}M^*)^*\simeq (\mathcal{H}\otimes_{\mathcal{A}}n_{w_2}M^*)^*$.

Now we have 
\begin{align*}
(\mathcal{H}\otimes_{\mathcal{A}}n_{w_1}M^*)^* & = \Hom_C(\mathcal{H}\otimes_{\mathcal{A}}n_{w_1}M^*,C)\\
& \simeq \Hom_{\mathcal{A}}(\mathcal{H},n_{w_1}M^{**}).
\end{align*}
Hence we have $\Hom_{\mathcal{A}}(\mathcal{H},n_{w_1}M^{**})\simeq \Hom_{\mathcal{A}}(\mathcal{H},n_{w_2}M^{**})$.
We have an embedding $M\hookrightarrow M^{**}$.
Let $L$ be the cokernel.
Then $\supp L = \Lambda^+(1)$ and  we have an embedding $L\hookrightarrow L^{**}$.
Therefore we have an exact sequence $0\to M\to M^{**}\to L^{**}$ and it gives $0\to n_{w_i}M\to n_{w_i}M^{**}\to n_{w_i}L^{**}$ for $i = 1,2$.
Hence we get the following commutative diagram with exact rows:
\[
\begin{tikzcd}
0\arrow[d] & 0\arrow[d]\\
\Hom_{\mathcal{A}}(\mathcal{H},n_{w_1}M)\arrow[d] & \Hom_{\mathcal{A}}(\mathcal{H},n_{w_2}M)\arrow[d]\\
\Hom_{\mathcal{A}}(\mathcal{H},n_{w_1}M^{**})\arrow[r,dash,"\sim"]\arrow[d] & \Hom_{\mathcal{A}}(\mathcal{H},n_{w_2}M^{**})\arrow[d]\\
\Hom_{\mathcal{A}}(\mathcal{H},n_{w_1}L^{**})\arrow[r,dash,"\sim"] & \Hom_{\mathcal{A}}(\mathcal{H},n_{w_2}L^{**}).
\end{tikzcd}
\]
We have $\Hom_{\mathcal{A}}(\mathcal{H},n_{w_1}M)\simeq \Hom_{\mathcal{A}}(\mathcal{H},n_{w_2}M)$.
\end{proof}
For given $w\in W$, set $J = \Delta\setminus\Delta_w$.
Then we have $\Delta_{w_J} = \Delta\setminus J = \Delta_{w}$.
Therefore, to prove Lemma~\ref{lem:injective, twist form}, we may assume that $w = w_J$ for some $J\subset\Delta$.

\subsection{Reduction to $w = w_\Delta$}\label{subsec:Reduction to w=w_Delta}
Set 
\begin{equation}\label{eq:definition of A_w}
\mathcal{A}_w = \bigoplus_{\lambda\in w(\Lambda^+(1))}CE(\lambda)\subset \mathcal{A}.
\end{equation}
\begin{lem}\label{lem:Switch from A to A_w}
Let $M$ be an $\mathcal{A}$-module such that $\supp M = w(\Lambda^+(1))$.
Then we have $\Hom_{\mathcal{A}}(\mathcal{H},M)\xrightarrow{\sim}\Hom_{\mathcal{A}_w}(\mathcal{H},M)$.
\end{lem}
\begin{proof}
Let $\varphi\colon \mathcal{H}\to M$ be an $\mathcal{A}_w$-module homomorphism and we prove that $\varphi$ is $\mathcal{A}$-equivariant.
Fix $\lambda_0\in \Lambda(1)$ such that $w^{-1}(\nu(\lambda_0))$ is dominant and regular.
Since $\supp M = w(\Lambda^+(1))$, $E(\lambda_0)$ is invertible on $M$.
For $\mu\in \Lambda(1)$ such that $w^{-1}(\nu(\mu))$ is not dominant, we have $E(\mu)E(\lambda_0) = 0$.
Hence for $X\in \mathcal{H}$, we have $\varphi(XE(\mu)) = E(\lambda_0)^{-1}\varphi(XE(\mu)E(\lambda_0)) = 0$.
Since $E(\mu) = 0$ on $n_wM$, $E(\mu)\varphi(X) = 0$.
Hence we get $\varphi(XE(\mu)) = 0 = E(\mu)\varphi(X)$.
Therefore $\varphi$ is $\mathcal{A}$-equivariant.
\end{proof}
An element $E(\lambda)$ belongs to
\begin{itemize}
\item $\mathcal{A}_w$ if $\langle \nu(\lambda),w(\alpha)\rangle\ge 0$ for any $\alpha\in\Sigma^+$.
\item $j_J^{-*}(\mathcal{H}_J^-\cap \mathcal{A}_J)$ if $\langle \nu(\lambda),\alpha\rangle \ge 0$ for any $\alpha\in\Sigma^+\setminus\Sigma^+_J$.
\end{itemize}
(The second one follows from the following fact and \cite[Lemma~2.6]{arXiv:1612.01312}: a basis of $\mathcal{H}_J^-\cap \mathcal{A}_J$ is given by $\{E(\lambda)\}$ where $\lambda$ runs through as above~\cite[Lemma~4.2]{arXiv:1406.1003_accepted}.)
Since $w_J(\Sigma^+) = \Sigma^-_J\cup(\Sigma^+\setminus\Sigma_J^+)\supset \Sigma^L\setminus\Sigma_J^+$, we have $\mathcal{A}_{w_J}\subset j_J^{-*}(\mathcal{H}_J^-\cap \mathcal{A}_J)$.
Hence we have
\begin{align*}
\Hom_{\mathcal{A}_{w_J}}(\mathcal{H},n_{w_J}M) & \simeq \Hom_{\mathcal{A}_{w_J}}(\mathcal{H}\otimes_{j_J^{-*}(\mathcal{H}_J^-)}j_J^{-*}(\mathcal{H}_J^-),n_{w_J}M)\\
& \simeq \Hom_{(\mathcal{H}_J^-,j_J^{-*})}(\mathcal{H},\Hom_{\mathcal{A}_{w_J}}(\mathcal{H}_J^-,n_{w_J}M)).
\end{align*}
Since $j_J^{-*}(\mathcal{H}_J^-\cap \mathcal{A}_J)$ contains $\mathcal{A}_{w_J}$, we have $\mathcal{A}_{w_J}\hookrightarrow \mathcal{H}_J^-\cap \mathcal{A}_J\hookrightarrow \mathcal{A}_J$.
More precisely, $\mathcal{A}_{w_J}\hookrightarrow \mathcal{A}_{J,w_J}$ via $E(\lambda)\mapsto E^J(\lambda)$.
(If $E(\lambda)\in \mathcal{A}_{w_J}$, then $w_J^{-1}(\nu(\lambda))$ is dominant with respect to $\Delta$, hence it is also dominant with respect to $J$. Therefore $E^J(\lambda)\in \mathcal{A}_{J,w_J}$.)
\begin{lem}\label{lem:extends to A_J}
We regard $\mathcal{A}_{w_J}$ as a subalgebra of $\mathcal{A}_J$ via the above embedding.
Then $n_{w_J}M$ is uniquely extended to $A_J$, namely there exists a unique $\mathcal{A}_J$-module $M_J$ such that $\supp M_J = \Lambda^+(1)_J$ and $n_{w_J}M_J|_{\mathcal{A}_{w_J}} = n_{w_J}M|_{\mathcal{A}_{w_J}}$.
\end{lem}
\begin{proof}
First we prove that $n_{w_J}M$ is uniquely extended to $\mathcal{A}_{J,w_J}$.
Take $\lambda_0\in \Lambda_S(1)$ such that
\begin{itemize}
\item $\langle \nu(\lambda_0),\alpha\rangle = 0$ for all $\alpha\in\Sigma_J^+$.
\item $\langle \nu(\lambda_0),\alpha\rangle > 0$ for all $\alpha\in\Sigma^+\setminus \Sigma_J^+$.
\end{itemize}
Note that $w_J(\Sigma_J^+) = \Sigma_J^-$ and $w_J(\Sigma^+\setminus\Sigma_J^+) = \Sigma^+\setminus\Sigma_J^+$.
Hence we have $\lambda_0\in w_J(\Lambda^+(1))$, $E^J(\lambda_0)$ is central in $\mathcal{A}_{J,w_J}$ (since $\lambda_0\in \Lambda_S(1)$ is central in $\Lambda(1)$) and $E^J(\lambda_0)$ is invertible by the first condition and Remark~\ref{rem:inverse of E(lambda)}.
The embedding $\mathcal{A}_{w_J}\hookrightarrow\mathcal{A}_{J,w_J}$ induces $\mathcal{A}_{w_J}[E(\lambda_0)^{-1}]\hookrightarrow\mathcal{A}_{J,w_J}$.
We prove that this is surjective.
Let $E^J(\mu)\in \mathcal{A}_{J,w_J}$.
Then we have $\langle w_J(\nu(\mu)),\alpha\rangle\ge 0$ for any $\alpha\in\Sigma_J^+$.
Therefore, for sufficiently large $n\in\Z_{>0}$, we have $\lambda_0^n\mu\in w_J(\Lambda^+(1))$.
The elements $\nu(\lambda_0)$ and $\nu(\mu)$ are in the same closed Weyl chamber $w_J\nu(\Lambda^+(1)_J)$ with respect to $J$.
Hence $E^J(\lambda_0^n)E^J(\mu) = E^J(\lambda_0^n\mu)$ which is in the image of $\mathcal{A}_{w_J}\hookrightarrow \mathcal{A}_{J,w_J}$.
Therefore $\mathcal{A}_{w_J}[E(\lambda_0)^{-1}]\hookrightarrow\mathcal{A}_{J,w_J}$ is surjective.
Now we get the lemma since $E(\lambda_0)$ is invertible on $n_{w_J}M$. (Recall that $\supp n_{w_J}M = w_J(\Lambda^+(1))$ and $\lambda_0\in w_J(\Lambda^+(1))$.)

So we have the extension $N_J$ of $n_{w_J}M$ to $\mathcal{A}_{J,w_J}$.
Define the action of $E^J(\lambda)$ on $N_J$ by zero for $\lambda\in \Lambda(1)\setminus w_J(\Lambda^+(1)_J)$.
Then $N_J$ is an $\mathcal{A}_J$-module such that $\supp N_J = w_J(\Lambda^+(1)_J)$ which is desired.
From the definition of the support, this is the only way to extend the module $N_J$ to $\mathcal{A}_J$.
We get the lemma.
\end{proof}

Take $M_J$ as in the lemma.
We have
\begin{align*}
\Hom_{\mathcal{A}_{w_J}}(\mathcal{H},n_{w_J}M) 
& \simeq \Hom_{(\mathcal{H}_{J}^-,j_J^{-*})}(\mathcal{H},\Hom_{\mathcal{A}_{w_J}}(\mathcal{H}_J^-,n_{w_J}M))\\
& \simeq \Hom_{(\mathcal{H}_{J}^-,j_J^{-*})}(\mathcal{H},\Hom_{\mathcal{A}_{w_J}}(\mathcal{H}_J^-,n_{w_J}M_J)).
\end{align*}
\begin{lem}
The homomorphisms
\[
\Hom_{\mathcal{A}_J}(\mathcal{H}_J,n_{w_J}M_J)\to
\Hom_{\mathcal{A}_{w_J}}(\mathcal{H}_J,n_{w_J}M_J)\to
\Hom_{\mathcal{A}_{w_J}}(\mathcal{H}_J^-,n_{w_J}M_J)
\]
are both isomorphisms.
\end{lem}
\begin{proof}
The first one is isomorphism by the similar argument in the proof of Lemma~\ref{lem:Switch from A to A_w}.

Take $\lambda_0\in \Lambda(1)$ such that 
\begin{itemize}
\item $\lambda_0\in Z(W_J(1))$.
\item $\langle \nu(\lambda_0), \alpha\rangle > 0$ for any $\alpha\in\Sigma^+\setminus\Sigma^+_J$.
\end{itemize}
Then $\mathcal{H}_J = \mathcal{H}_J^-[E^J(\lambda_0)^{-1}]$ \cite[Proposition~2.5]{arXiv:1612.01312}.
Since $E^J(\lambda_0)$ is invertible in $\mathcal{A}_J$, it is also invertible on $n_{w_J}M_J$. (Note that $n_{w_J}M_J$ is an $\mathcal{A}_J$-module.)
Hence the second homomorphism is an isomorphism.
\end{proof}
Therefore we get
\[
\Hom_{\mathcal{A}}(\mathcal{H},n_{w_J}M)\simeq \Hom_{(\mathcal{H}_J^-,j_J^{-*})}(\mathcal{H},\Hom_{\mathcal{A}_J}(\mathcal{H}_J,n_{w_J}M_J)).
\]
\begin{lem}
Let $X$ be an $\mathcal{H}_J$-module and assume that $X\to X\otimes_{\mathcal{H}_J}\cInd_{I(1)_J}^{L_J}\trivrep$ is injective.
Then for $Y = \Hom_{(\mathcal{H}_J^-,j_J^{-*})}(\mathcal{H},X)$, $Y\to Y\otimes_{\mathcal{H}}\cInd_{I(1)}^G\trivrep$ is also injective.
\end{lem}

Therefore for the proof of Lemma~\ref{lem:injective, twist form}, it is sufficient to prove that
\[
\Hom_{\mathcal{A}_J}(\mathcal{H}_J,n_{w_J}M_J)\to
\Hom_{\mathcal{A}_J}(\mathcal{H}_J,n_{w_J}M_J)\otimes_{\mathcal{H}_J}\cInd_{I(1)_J}^{L_J}\trivrep
\]
is injective, namely we may assume that $w = w_\Delta$.

\begin{proof}
Set $J' = -w_\Delta(J)$ and put $n = n_{w_\Delta}n_{w_J}$.
Then $l\mapsto nln^{-1}$ gives an isomorphism $L_J\to L_{J'}$ and sends $I(1)_J$ to $I(1)_{J'}$.
Therefore it induces an isomorphism $\mathcal{H}_J\to \mathcal{H}_{J'}$.
Define an $\mathcal{H}_{J'}$-module $X'$ as the pull-back of $X$ by this isomorphism (see \cite{arXiv:1406.1003_accepted}).
Then $X\to X\otimes_{\mathcal{H}_J}\cInd_{I(1)_J}^{L_J}\trivrep$ induces $X'\to X'\otimes_{\mathcal{H}_{J'}}\cInd_{I(1)_{J'}}^{L_{J'}}\trivrep$ and the latter map is also injective.
By \cite[Proposition~4.15]{arXiv:1406.1003_accepted}, we have $Y\simeq X'\otimes_{(\mathcal{H}_{J'},j_{J'}^+)}\mathcal{H}$.
By \cite[Proposition~4.1]{MR3437789}, the functor $(\cdot)\otimes_{(\mathcal{H}_{J'}^+,j_{J'}^+)}\mathcal{H}$ is exact.
Hence, using the assumption in the lemma, the map
\[
Y \simeq X'\otimes_{(\mathcal{H}_{J'}^+,j_{J'}^+)}\mathcal{H}\to (X'\otimes_{\mathcal{H}_{J'}}\cInd_{I(1)_{J'}}^{L_{J'}}\trivrep)^{I(1)_{J'}}\otimes_{(\mathcal{H}_{J'}^+,j_{J'}^+)}\mathcal{H}
\]
is injective.
By \cite[Proposition~4.4]{arXiv:1703.04921}
\[
(X'\otimes_{\mathcal{H}_{J'}}\cInd_{I(1)_{J'}}^{L_{J'}}\trivrep)^{I(1)_{J'}}\otimes_{(\mathcal{H}_{J'}^+,j_{J'}^+)}\mathcal{H}\simeq
(\Ind_{P_{J'}}(X'\otimes_{\mathcal{H}_{J'}}\cInd_{I(1)_{J'}}^{L_{J'}}\trivrep))^{I(1)}.
\]
In particular, 
\[
(X'\otimes_{\mathcal{H}_{J'}}\cInd_{I(1)_{J'}}^{L_{J'}}\trivrep)^{I(1)_{J'}}\otimes_{(\mathcal{H}_{J'}^+,j_{J'}^+)}\mathcal{H}\to
\Ind_{P_{J'}}(X'\otimes_{\mathcal{H}_{J'}}\cInd_{I(1)_{J'}}^{L_{J'}}\trivrep)
\]
is injective.
Finally, by \cite[Corollary~4.7.]{arXiv:1703.04921},
\[
\Ind_{P_{J'}}(X'\otimes_{\mathcal{H}_{J'}}\cInd_{I(1)_{J'}}^{L_{J'}}\trivrep)\simeq
Y\otimes_{\mathcal{H}}\cInd_{I(1)}^G\trivrep.
\]
Combining all of these, we conclude the lemma.
\end{proof}

\subsection{Some more reductions}
We note the following.
\begin{itemize}
\item $j_\emptyset^+(\mathcal{H}_\emptyset^+) = \mathcal{A}_{w_\Delta}$.
\item $j_\emptyset^{-*}(\mathcal{H}_\emptyset^-) = \mathcal{A}_{1}$.
\end{itemize}
This follows from the definition of $\mathcal{H}_+$, $\mathcal{H}_-$ and \cite[Lemma~2.6]{arXiv:1612.01312}.
See the argument in \ref{subsec:Reduction to w=w_Delta}.
By these identities, we regard $\mathcal{A}_1$ and $\mathcal{A}_{w_\Delta}$ as a subalgebra of $\mathcal{H}_\emptyset = \mathcal{A}_\emptyset$.

By Lemma~\ref{lem:Switch from A to A_w}, we have $\Hom_{\mathcal{A}}(\mathcal{H},n_{w_\Delta}M)\simeq \Hom_{\mathcal{A}_{w_\Delta}}(\mathcal{H},n_{w_\Delta}M)$.
By Lemma~\ref{lem:extends to A_J}, there exists an $\mathcal{A}_\emptyset$-module $M_\emptyset$ such that $M|_{\mathcal{A}_{1}}\simeq M_\emptyset|_{\mathcal{A}_{1}}$.
It is easy to see that $n_{w_\Delta}M|_{\mathcal{A}_{w_\Delta}}\simeq n_{w_\Delta}M_\emptyset|_{\mathcal{A}_{w_\Delta}}$.
We have
\begin{align*}
\Hom_{\mathcal{A}}(\mathcal{H},n_{w_\Delta}M) & \simeq
\Hom_{(\mathcal{H}_\emptyset^+,j_\emptyset^+)}(\mathcal{H},n_{w_\Delta}M_\emptyset)\\
& \simeq \Hom_{(\mathcal{H}_\emptyset^-,j_\emptyset^-)}(\mathcal{H},M_\emptyset)\tag*{(\cite[Proposition~4.13]{arXiv:1612.01312})}\\
& \simeq M_\emptyset\otimes_{(\mathcal{H}_\emptyset^-,j_\emptyset^{-*})}\mathcal{H}\tag*{(\cite[Corollary~4.19]{arXiv:1612.01312})}\\
& = M\otimes_{\mathcal{A}_1}\mathcal{H}\tag*{($j_\emptyset^{-*}(\mathcal{H}_\emptyset^-) = \mathcal{A}_1$)}.
\end{align*}
We have $j_\emptyset^{-*}(\mathcal{H}_\emptyset^-) = \mathcal{A}_1$ by \cite[Lemma~2.6]{arXiv:1612.01312}.
\begin{lem}
$M\otimes_{\mathcal{A}_1}\mathcal{H}\simeq M\otimes_{\mathcal{A}}\mathcal{H}$.
\end{lem}
\begin{proof}
The same proof as \cite[Lemma~4.29]{arXiv:1406.1003_accepted} can apply.
\end{proof}

Hence we get $\Hom_{\mathcal{A}}(\mathcal{H},n_{w_\Delta}M)\simeq M\otimes_{\mathcal{A}}\mathcal{H}$.
Therefore it is sufficient to prove the following.
\begin{lem}\label{lem:injective, by tensor}
Let $M$ be an $\mathcal{A}$-module such that $\supp M = \Lambda^+(1)$.
Then $M\otimes_{\mathcal{A}}\mathcal{H}\to M\otimes_{\mathcal{A}}\cInd_{I(1)}^G\trivrep$ is injective.
\end{lem}

The group algebra $C[Z_\kappa]$ is a subalgebra of $\mathcal{A}$ via the map $t\mapsto T_t = E(t)$ for $t\in Z_\kappa$.
Let $\widehat{Z}_\kappa$ denote the set of characters of $Z_\kappa$.
Since the order of $Z_\kappa$ is prime to $p$, $M$ is semisimple as a $C[Z_\kappa]$-module.
Let $\psi\in Z_\kappa$ and set $M_\psi = \{m\in M\mid mT_t = \psi(t)m\ (t\in Z_\kappa)\}$.
Since $Z_\kappa$ is normal in $\Lambda(1)$, the conjugate action of $\Lambda(1)$ on $Z_\kappa$ induces the action on $\widehat{Z}_\kappa$.
The formula $E(\lambda)T_t = T_{\lambda t\lambda^{-1}}E(\lambda)$ implies that $M_\psi E(\lambda) \subset M_{\lambda^{-1}(\psi)}$.
For an orbit $\omega$ of this action in $\widehat{Z}_\kappa$, we put $M_\omega = \bigoplus_{\psi\in\omega}M_\psi$.
Then $M_\omega$ is stable under the $\mathcal{A}$-action and we have $M = \bigoplus_\omega M_\omega$.
Therefore we may assume that $M = M_\omega$ for some $\omega$ to prove Lemma~\ref{lem:injective, by tensor}.

Let $\alpha\in\Delta$ and consider the image of $Z\cap L_{\{\alpha\}}\cap G'$ in $\Lambda(1)$.
We denote this subgroup by $\Lambda'_\alpha(1)$.
Consider the following condition: $\psi$ is trivial on $Z_\kappa\cap \Lambda'_\alpha(1)$.
Since $Z_\kappa\cap \Lambda'_\alpha(1)$ is normal in $\Lambda(1)$, for $t\in Z_\kappa\cap \Lambda'_\alpha(1)$ and $\lambda\in \Lambda(1)$, we have $(\lambda\psi)(t) = \psi(\lambda^{-1}t\lambda) = 1$ if $\psi$ satisfies this condition.
Hence this condition only depends on $\Lambda(1)$-orbit.

Assume that $\omega$ is a $\Lambda(1)$-orbit in $\widehat{Z}_\kappa$ and we also assume that $\psi$ is not trivial on $Z_\kappa\cap \Lambda'_\alpha(1)$ for some (equivalently any) $\psi\in\omega$.
Then by \cite[Theorem~3.13]{arXiv:1406.1003_accepted}, we have $M\otimes_{\mathcal{A}}\mathcal{H}\simeq n_{s_\alpha}M\otimes_{\mathcal{A}}\mathcal{H}$.
In this case, as we have seen before, Lemma~\ref{lem:injective, by tensor} follows from Lemma~\ref{lem:injective, by tensor} for a proper Levi subgroup.
Therefore we may assume that there is no such $\alpha$ by induction on $\dim G$.
Hence it is sufficient to prove the following to prove Lemma~\ref{lem:injective, by tensor}.
\begin{lem}\label{lem:injective lemma, final form}
Let $M$ be an $\mathcal{A}$-module such that $\supp(M) = \Lambda^+(1)$ and $Z_\kappa\cap \Lambda'_\alpha(1)$ acts trivially on $M$ for all $\alpha\in\Delta$.
Then $M\otimes_{\mathcal{A}}\mathcal{H}\to M\otimes_{\mathcal{A}}\cInd_{I(1)}^G\trivrep$ is injective.
\end{lem}

\subsection{Hecke modules}
As discussed in \ref{subsec:Reduction to w=w_J}, we have the following
\[
M\otimes_{\mathcal{A}}\mathcal{H}\simeq M\otimes_{C[\Lambda(1)]}(C[\Lambda(1)]\otimes_{\mathcal{A}}\mathcal{H}),
\]
We decompose this module along the action of $Z_\kappa$.

Set $C[\Lambda(1)]_\psi = \{f\in C[\Lambda(1)]\mid \tau_tf = \psi(t)f\ (t\in Z_\kappa)\}$ and for a $\Lambda(1)$-stable subset $\omega\subset\widehat{Z}_\kappa$ we put $C[\Lambda(1)]_\omega = \bigoplus_{\psi\in \omega}C[\Lambda(1)]_\psi$.
From the definition, it is obvious that $C[\Lambda(1)]$ is invariant under the right action of $C[\Lambda(1)]$.
\begin{lem}\label{lem:decompotision of the group ring by Z_k action}
We have $C[\Lambda(1)]_\omega = \bigoplus_{\psi\in\omega}\{f\in C[\Lambda(1)]\mid f\tau_t = \psi(t)f\ (t\in Z_\kappa)\}$.
\end{lem}
\begin{proof}
Let $\psi\in\omega$, $f\in C[\Lambda(1)]_\psi$ and we write $f = \sum_{\lambda\in \Lambda(1)}c_\lambda \tau_\lambda$ where $c_\lambda\in C$.
Set $e = \#Z_\kappa^{-1}\sum_{t\in Z_\kappa}\psi(t)^{-1}\tau_t\in C[Z_\kappa]$.
Then $ef = f$ and $e\tau_t = \psi(t)e$ for each $t\in Z_\kappa$.
We prove $e\tau_\lambda\in C[\Lambda(1)]_\omega$ for each $\lambda\in \Lambda(1)$.
We have $e\tau_\lambda\tau_t = e\tau_{\lambda t\lambda^{-1}}\tau_\lambda  = (\lambda^{-1}\psi)(t) e\tau_\lambda$.
Since $\lambda^{-1}\psi\in \omega$, we get the lemma.
\end{proof}
Therefore $C[\Lambda(1)]_\omega$ is a two-sided ideal of $C[\Lambda(1)]$.
Using $Z_\kappa$-action, some objects appearing here are decomposed.
Here is a list.
\begin{itemize}
\item $C[\Lambda(1)] = C[\Lambda(1)]_\omega\times C[\Lambda(1)]_{\widehat{Z}_\kappa\setminus\omega}$ as $C$-algebras.
\item $\mathcal{A} = \mathcal{A}_\omega\times\mathcal{A}_{\widehat{Z}_\kappa\setminus\omega}$ as $C$-algebras with the obvious notation.
\item The homomorphism \eqref{eq:hom chi tilde} induces $\mathcal{A}_\omega\to C[\Lambda(1)]_\omega$ and $\mathcal{A}_{\widehat{Z}_\kappa\setminus\omega}\to C[\Lambda(1)]_{\widehat{Z}_\kappa\setminus\omega}$.
\end{itemize}
Let $M$ be an $\mathcal{A}$-module such that $\supp M = \Lambda^+(1)$ and $M = M_\omega$.
Then as in \ref{subsec:Reduction to w=w_Delta}, $M$ is a $C[\Lambda(1)]$-module and this action factors through $C[\Lambda(1)]\to C[\Lambda(1)]_\omega$.
Hence we have
\begin{equation}\label{eq:M as tensor with universal module}
M\otimes_{\mathcal{A}}\mathcal{H}\simeq M\otimes_{C[\Lambda(1)]_\omega}(C[\Lambda(1)]_\omega\otimes_{\mathcal{A}}\mathcal{H})
\end{equation}

In \cite[Section~3]{arXiv:1406.1003_accepted}, it is proved that, for any $w\in W_0$, $1\otimes 1\mapsto 1\otimes T_{n_{w_\Delta w^{-1}}}^*$ gives a homomorphism
\[
n_wC[\Lambda(1)]\otimes_{\mathcal{A}}\mathcal{H}\to n_{w_\Delta}C[\Lambda(1)]\otimes_{\mathcal{A}}\mathcal{H}
\]
which is injective \cite[Proposition~3.12]{arXiv:1406.1003_accepted}.
This is a $(C[\Lambda(1)],\mathcal{H})$-bimodule homomorphism.
The homomorphism is compatible with the decomposition $n_wC[\Lambda(1)]\otimes_{\mathcal{A}}\mathcal{H}\simeq n_wC[\Lambda(1)]_\omega\otimes_{\mathcal{A}}\mathcal{H}\oplus n_wC[\Lambda(1)]_{\widehat{Z}_\kappa\setminus\omega}\otimes_{\mathcal{A}}\mathcal{H}$.
Hence we get the homomorphism
\[
n_wC[\Lambda(1)]_\omega\otimes_{\mathcal{A}}\mathcal{H} \to n_{w_\Delta}C[\Lambda(1)]_\omega\otimes_{\mathcal{A}}\mathcal{H}
\]
which is again injective.
By \cite[Theorem~3.13]{arXiv:1406.1003_accepted}, the image of this homomorphism only depends on $\Delta_w$.
Let $X_J$ be the image of this homomorphism where $J = \Delta_w$.
This is a $(C[\Lambda(1)]_\omega,\mathcal{A})$-module.
We have $M\otimes_{\mathcal{A}}\mathcal{H} = M\otimes_{C[\Lambda(1)]_\omega}X_\Delta$ by \eqref{eq:M as tensor with universal module}.

\begin{lem}
If $J'\supset J$, then $X_{J'}\subset X_J$.
\end{lem}
\begin{proof}
Note that $\Delta_{w_\Delta w_J} = J$.
Hence by the definition, $X_J$ is a submodule in $n_{w_\Delta}C[\Lambda(1)]\otimes_{\mathcal{A}}\mathcal{H}$ generated by $1\otimes T^*_ {n_{w_\Delta w_J w_\Delta}}$.
If $J'\supset J$, then $\ell(w_\Delta w_{J}w_{J'}w_\Delta) = \ell(w_{J}w_{J'}) = \ell(w_{J'}) - \ell(w_{J}) = \ell(w_\Delta w_{J'}w_\Delta) - \ell(w_\Delta w_J w_\Delta)$.
Hence $T^*_{n_{w_\Delta w_{J'}w_\Delta}} = T^*_{n_{w_\Delta w_Jw_\Delta}}T^*_{n_{w_\Delta w_Jw_{J'}w_\Delta}}$.
Therefore $1\otimes T^*_{n_{w_\Delta w_{J'}w_\Delta}}\in X_J$.
Since $X_{J'}$ is generated by $1\otimes T^*_{n_{w_\Delta w_{J'}w_\Delta}}$, we have $X_{J'}\subset X_J$.
\end{proof}

\begin{lem}\label{lem:X_J in C}
$X_J\in \mathcal{C}$.
\end{lem}
\begin{proof}
Take $\lambda\in \Lambda_S(1)$ such that $\nu(\lambda)$ is regular dominant.
Then we have $z_\lambda = \sum_{v\in W_0}E(n_v\lambda n_v^{-1})$ by Lemma~\ref{lem:description of z_lambda}.
Let $f\otimes X\in X_w$.
Then, since $z_\lambda$ is in the center, we have $(f\otimes X)z_\lambda = f\otimes z_\lambda X = f\otimes \sum_{v\in W_0}E(n_v\lambda n_v^{-1})X = f\tau_\lambda\otimes X$ in $n_{w_\Delta}C[\Lambda(1)]_\omega\otimes_{\mathcal{A}}\mathcal{H}$.
Since $f\mapsto f\tau_\lambda$ is invertible, $z_\lambda$ is invertible on $X_w$.
\end{proof}

Note that $n_{w_\Delta}C[\Lambda(1)]_\omega\otimes_{\mathcal{A}}\mathcal{H} \simeq n_{w_\Delta}C[\Lambda(1)]_\omega\otimes_{\mathcal{A}_{w_\Delta}}\mathcal{H}$ \cite[Proposition~3.12]{arXiv:1406.1003_accepted}.
Hence $X_\emptyset = n_{w_\Delta}C[\Lambda(1)]_\omega\otimes_{(\mathcal{H}_\emptyset^+,j_\emptyset^+)}\mathcal{H}$.
This is a parabolically induced module~\cite{MR3437789}.
By \cite{MR3437789}, we have $n_{w_\Delta}C[\Lambda(1)]_\omega\otimes_{\mathcal{A}}\mathcal{H} = \bigoplus_{w\in W_0}n_{w_\Delta}C[\Lambda(1)]_\omega\otimes T_{n_w}$.
Since $T_{n_w}^* \in T_{n_w} + \sum_{v < w}C[Z_\kappa]T_{n_v}$, we have $n_{w_\Delta}C[\Lambda(1)]_\omega\otimes_{\mathcal{A}}\mathcal{H}  = \bigoplus_{w\in W_0}n_{w_\Delta}C[\Lambda(1)]_\omega\otimes T_{n_w}^*$.

Set $Y_w = n_{w_\Delta}C[\Lambda(1)]_\omega\otimes T_{n_w}^*\subset X_\emptyset$.
Then the subspace $Y_w$ is the image of $n_wC[\Lambda(1)]_\omega\otimes 1$ by the above injective homomorphism.
In particular, $Y_w$ is $\mathcal{A}$-stable and isomorphic to $n_wC[\Lambda(1)]_\omega$.
We have $X_\emptyset = \bigoplus_{w\in W_0} Y_w$.
This is the decomposition in Lemma~\ref{lem:decomposition along support}.
By the functoriality of the decomposition, we have $X_J = \bigoplus_{w\in W_0}(X_J\cap Y_w)$.

\subsection{Representations of $G$}
Recall that we have fixed a special parahoric subgroup $K$.
Irreducible representations $V$ of $K$ are parametrized by a pair $(\psi,J)$ where $\psi$ is a character of $Z_\kappa$ and $J$ a certain subset of $\Delta$.
Here for $V$, $\psi$ and $J$ is given by the following: $\psi \simeq V^{I(1)}$ and $W_{0,J} = \Stab_{W_0}(V^{I(1)})$.
Let $V_{\psi,J}$ be the irreducible representation of $K$ which corresponds to $(\psi,J)$ and put $V_J = \bigoplus_{\psi\in \omega}V_{\psi^{-1},J}$.
In the rest of this paper, we fix a basis of $V_{\psi^{-1},J}^{I(1)}$ for each $\psi$ and $J$.
\begin{lem}\label{lem:Hecke algebra and the group ring}
\begin{enumerate}
\item The Hecke algebra $\End_Z(\cInd_{Z\cap K}^ZV_J^{I(1)})$ is isomorphic to $C[\Lambda(1)]_\omega$.
\item We have the Satake homomorphism 
\[
\End_G(\cInd_K^GV_J)\hookrightarrow \End_Z(\cInd_{Z\cap K}^ZV_J^{I(1)})\simeq C[\Lambda(1)]_\omega
\]
and its image is $C[\Lambda^+(1)]_\omega$.
\end{enumerate}
\end{lem}
\begin{proof}
Let $\mathcal{H}(\psi_1^{-1},\psi_2^{-1})$ is the space of functions $\varphi\colon Z\to C$ such that $\supp\varphi$ is compact and $\varphi(t_1zt_2) = \psi_1^{-1}(t_1)\varphi(z)\psi_2^{-1}(t_2)$ for any $z\in Z$ and $t_1,t_2\in Z\cap K$.
Since $V_J^{I(1)} \simeq \bigoplus_{\psi\in \omega}\psi^{-1}$, a standard argument for Hecke algebras implies
\begin{align*}
\End_Z(\Ind_{Z\cap K}^ZV_J^{I(1)}) & \simeq \bigoplus_{\psi_1,\psi_2\in\omega}\Hom_Z(\cInd_{Z\cap K}^Z\psi_1^{-1},\cInd_{Z\cap K}^Z\psi_2^{-1})\\
& \simeq \bigoplus_{\psi_1,\psi_2\in \omega}\mathcal{H}(\psi_1^{-1},\psi_2^{-1}).
\end{align*}
This space is a subalgebra of $\mathcal{H}_Z$ where $\mathcal{H}_Z$ is the functions $\varphi$ on $Z$ which is invariant under the left (and equivalently right) multiplication by $Z\cap I(1)$ and whose support is compact.
The homomorphism $\varphi\mapsto \sum_{z\in Z/(Z\cap K)}\varphi(z)\tau_z$ gives an isomorphism $\mathcal{H}_Z\simeq C[\Lambda(1)]$.
As a subspace of both sides, it is easy to see that we get the desired isomorphism.

The Satake transform
\[
\Hom_G(\cInd_K^GV_{\psi_1,J},\cInd_{K}^GV_{\psi_2,J})
\to
\Hom_Z(\cInd_{Z\cap K}^Z\psi_1^{-1},\cInd_{Z\cap K}^Z\psi_2^{-1})
\]
is defined in \cite[2]{MR3001801} and the image is described in \cite[Theorem~1.1]{arXiv:1805.00244}.
\end{proof}
\begin{rem}
In the identification (1) in the lemma, we need to fix an isomorphism $V_J^{I(1)}\simeq \bigoplus_{\psi\in \omega}\psi^{-1}$.
We use our fixed basis of $V_{\psi^{-1},J}^{I(1)}$ for this isomorphism.
\end{rem}
By the lemma, $C[\Lambda^+(1)]_\omega$ acts on $\cInd_K^GV_J$.
Define a representation $\pi_J$ of $G$ by $\pi_J = C[\Lambda(1)]_\omega\otimes_{C[\Lambda^+(1)]_\omega}\cInd_K^GV_J$.
We prove $\pi_J^{I(1)} \simeq X_J$.

Recall that the $\mathcal{H}$-module $(\cInd_K^GV_J)^{I(1)}$ is described as follows.
Let $\mathcal{H}_{\mathrm{f}}$ be the Hecke algebra attached to the pair $(K,I(1))$.
Then $V_J^{I(1)}$ is naturally a right $\mathcal{H}_{\mathrm{f}}$-module and the algebra $\mathcal{H}_{\mathrm{f}}$ is a subalgebra of $\mathcal{H}$ with a basis $\{T_w\mid w\in W_0(1)\}$ where $W_0(1)$ is the inverse image of $W_0\subset W$ in $W(1)$.
Then we have $(\cInd_K^GV_J)^{I(1)}\simeq V_J^{I(1)}\otimes_{\mathcal{H}_{\mathrm{f}}}\mathcal{H}$~\cite{MR3646282}.
\begin{rem}
In the argument below, we will use results in \cite{MR3666048}.
In \cite{MR3666048}, we study an $\mathcal{H}_{\mathrm{f}}$-module denoted by $\eta^J = \bigoplus_{\psi\in\widehat{Z}_\kappa}V_{\psi,J}^{I(1)}$.
Using a similar argument in \cite{MR3666048} (or taking a direct summand of results), results are also true for an $\mathcal{H}_{\mathrm{f}}$-module $V_J^{I(1)}$.
\end{rem}
We have an action of $C[\Lambda^+(1)]_\omega$ on $V_J^{I(1)}\otimes_{\mathcal{H}_{\mathrm{f}}}\mathcal{H}$ \cite[Proposition~3.4]{MR3666048} and the above isomorphism $(\cInd_K^GV_J)^{I(1)}\simeq V_J^{I(1)}\otimes_{\mathcal{H}_{\mathrm{f}}}\mathcal{H}$ is $C[\Lambda^+(1)]_\omega$-equivariant. (This can be proved by the same argument in the proof of \cite[Proposition~5.1]{MR3666048}.)
Therefore we have
\[
\pi_J^{I(1)}\simeq C[\Lambda(1)]_\omega\otimes_{C[\Lambda^+(1)]_\omega}V_J^{I(1)}\otimes_{\mathcal{H}_{\mathrm{f}}}\mathcal{H}
\]
By \cite[Proposition~3.11]{MR3666048}, we have an isomorphism $C[\Lambda(1)]_{\omega}\otimes_{C[\Lambda^+(1)]_\omega}V_J^{I(1)}\otimes_{\mathcal{H}_{\mathrm{f}}}\mathcal{H}\simeq X_J$.
Hence $\pi_J^{I(1)}\simeq X_J$.

Therefore we have an embedding $X_J\to \pi_J$.
This homomorphism factors through $X_J\to X_J\otimes_{\mathcal{H}}\cInd_{I(1)}^G\trivrep$.
Recall that $M\otimes_{C[\Lambda(1)]} X_\Delta\simeq M\otimes_{\mathcal{A}}\mathcal{H}$ \eqref{eq:M as tensor with universal module}.
Hence for Lemma~\ref{lem:injective lemma, final form}, it is sufficient to prove that $M\otimes_{C[\Lambda(1)]}X_\Delta\to M\otimes_{C[\Lambda(1)]}\pi_\Delta$ is injective.

We have an isomorphism $\pi_\emptyset\simeq \Ind_{\overline{B}}^G(\cInd_{Z\cap K}^ZV_J^{I(1)})$ \cite{MR3001801}.
(To be precisely, the direct sum of a result in \cite{MR3001801}.)
An injective embedding $\pi_J\to \Ind_{\overline{B}}^G(\cInd_{Z\cap K}^ZV_J^{I(1)})\simeq \pi_\emptyset$ was given in \cite[Definition~7.1]{MR3001801}.
Hence we have a diagram
\[
\begin{tikzcd}
X_J \arrow[r]\arrow[d] & X_\emptyset\arrow[d]\\
\pi_J\arrow[r] & \pi_\emptyset.
\end{tikzcd}
\]
When $J = \emptyset$, $X_J\to X_\emptyset$ and $\pi_J\to \pi_\emptyset$ are both identities.
Hence this diagram is commutative.
\begin{lem}
This diagram is commutative for any $J$.
\end{lem}
\begin{proof}
Fix $\psi^{-1}\in \omega$.
It is sufficient to prove that the following diagram is commutative:
\begin{equation}\label{eq:commutative diagram for each psi}
\begin{tikzcd}
V_{\psi,J}^{I(1)}\otimes_{\mathcal{H}_\mathrm{f}}\mathcal{H}\arrow[r]\arrow[d] & n_{w_\Delta}C[\Lambda(1)]_\psi\otimes_{\mathcal{A}}\mathcal{H}\arrow[d]\\
\cInd_K^GV_{\psi,J}\arrow[r] & \Ind_{\overline{B}}^G(\cInd_{Z\cap K}^ZV_{\psi,\emptyset}^{I(1)}).
\end{tikzcd}
\end{equation}
Note that this diagram is commutative when $J = \emptyset$.

Let $v_0\in V_{\psi,J}^{I(1)}$ be our fixed basis.
Define $\varphi_J\in\cInd_K^GV_{\psi,J}$ by $\supp \varphi_J = K$ and $\varphi_J(1) = v_0$.
Then the map $V_{\psi,J}^{I(1)}\otimes_{\mathcal{H}_f}\mathcal{H}\to \cInd_K^G(V_{\psi,J})$ is given by $v_0\otimes 1\mapsto \varphi_J$.
Define $f_0\in\Ind_{\overline{B}}^G(\cInd_{Z\cap K}^ZV_{\psi,\emptyset}^{I(1)})$ by $f_0$ is $I(1)$-invariant, $\supp f_0 = \overline{B}n_{w_0}I(1)$, $\supp f_0(n_{w_\Delta}^{-1}) = Z\cap K$ and $f_0(n_{w_\Delta}^{-1})(1) = v_0$.
Then the function corresponding to $\varphi_{\emptyset}$ under $\cInd_K^GV_\emptyset\to \Ind_{\overline{B}}^G(\cInd_{Z\cap K}^ZV_\emptyset^{I(1)})$ is $f_0T_{n_{w_\Delta}}$ \cite[IV.9 Proposition]{MR3600042}.

Set $w = w_\Delta w_J$.
Then $X_J = n_{w}C[\Lambda(1)]_\omega\otimes_{\mathcal{A}}\mathcal{H}$.
The homomorphism $V_{\psi,J}^{I(1)}\otimes_{\mathcal{H}_{\mathrm{f}}}\mathcal{H}\to n_wC[\Lambda(1)]_\psi\otimes_{\mathcal{A}}\mathcal{H}$ is given by $v_0\otimes 1\mapsto 1\otimes T_{n_w}$ \cite[Lemma~3.8, Lemma~3.10]{MR3666048}.
Therefore, combining the above description of $\cInd_K^GV_\emptyset\to \Ind_{\overline{B}}^G(\cInd_{Z\cap K}^ZV_\emptyset^{I(1)})$, the homomorphism $n_{w_\Delta}C[\Lambda(1)]_\psi\otimes_{\mathcal{A}}\mathcal{H}\to \Ind_{\overline{B}}^G(\cInd_{Z\cap K}^ZV_{\psi,\emptyset}^{I(1)})$ is given by $1\otimes 1\mapsto f_0$.
(Recall that the map makes \eqref{eq:commutative diagram for each psi} commutative for $J = \emptyset$.)

We consider the image of $a = v_0\otimes 1\in V_{\psi,J}^{I(1)}\otimes_{\mathcal{H}_{\mathrm{f}}}\mathcal{H}$ in $\Ind_{\overline{B}}^G(\cInd_{Z\cap K}^ZV_{\psi,\emptyset}^{I(1)})$ in the two ways.
The image of $a$ in $n_{w_\Delta}C[\Lambda(1)]_\psi\otimes_{\mathcal{A}}\mathcal{H}$ is $1\otimes T_{n_{w_\Delta w^{-1}}}^*T_{n_{w}}$ by \cite[Proposition~3.11]{MR3666048} and the definition of $X_J\to X_\emptyset$.
Therefore the image of $a$ under $V_{\psi,J}^{I(1)}\otimes_{\mathcal{H}_{\mathrm{f}}}\mathcal{H}\to n_{w_\Delta}C[\Lambda(1)]_\psi\otimes_{\mathcal{A}}\mathcal{H}\to \Ind_{\overline{B}}^G(\cInd_{Z\cap K}^ZV_{\psi,\emptyset}^{I(1)})$ is $f_0T_{n_{w_\Delta w^{-1}}}^*T_{n_w}$.

By \cite[IV.9 Proposition]{MR3600042} (for $J = \Delta$), we have $f_0T_{n_{w_\Delta w^{-1}}}^* = \sum_{v\le w_\Delta w^{-1}}f_0T_{n_v}$.
Since $w_\Delta w^{-1} = w_\Delta w_Jw_\Delta$, $\{v\in W_0\mid v\le w_\Delta w^{-1}\} = w_\Delta W_{0,J}w_\Delta$.
Hence $f_0T_{n_{w_\Delta w^{-1}}}^*T_{n_w} = \sum_{v\in W_{J,0}}f_0T_{n_{w_\Delta v w_\Delta}}T_{n_{w_\Delta w_J}}$.
We have $\ell(w_\Delta v w_\Delta\cdot w_\Delta w_J) = \ell(w_\Delta vw_J) = \ell(w_\Delta) - \ell(v w_J) = \ell(w_\Delta) - \ell(w_J) + \ell(v) = \ell(w_\Delta w_J) + \ell(w_\Delta vw_\Delta)$.
Hence $T_{n_{w_\Delta v w_\Delta}}T_{n_{w_\Delta w_J}} = T_{n_{w_\Delta v w_J}}$.
Therefore, replacing $v$ with $v w_J$, we get $f_0T_{n_{w_\Delta w^{-1}}}^*T_{n_w} = \sum_{v\in W_{J,0}}f_0T_{w_\Delta v}$.
This is the image of $\varphi_J$ in $\Ind_{\overline{B}}^G(\cInd_{Z\cap K}^ZV_{\psi,\emptyset}^{I(1)})$ by \cite[IV.7 Corollary]{MR3600042}.
Hence the diagram \eqref{eq:commutative diagram for each psi} is commutative if we start with $a$.
Since the element $a$ generates $V_{\psi,J}^{I(1)}\otimes_{\mathcal{H}_{\mathrm{f}}}\mathcal{H}$ as an $\mathcal{H}$-module, the diagram \eqref{eq:commutative diagram for each psi} is commutative.
\end{proof}
Therefore we may regard $\pi_J$ and $X_J$ as a subspace of $\pi_\emptyset$.
We have $\pi_\emptyset \simeq \Ind_{\overline{B}}^G(\cInd_{Z\cap K}^ZV_J^{I(1)})$.
By the same argument in the proof of Lemma~\ref{lem:Hecke algebra and the group ring}, we have $\cInd_{Z\cap K}^ZV_J^{I(1)}\simeq C[\Lambda(1)]_\omega$.
Here again we use our fixed basis.
Hence we have $\pi_\emptyset\simeq \Ind_{\overline{B}}^G C[\Lambda(1)]_\omega$.
We identify $\pi_J$ with the image in $\Ind_{\overline{B}}^G C[\Lambda(1)]_\omega$.

\begin{rem}\label{rem:I-invariant and parabolic induction}
By \cite[IV.7. Proposition]{MR3600042} and the decomposition $G = \bigcup_{w\in W_0}\overline{B}n_wI(1)$ implies that $(\Ind_{\overline{B}}^GC[\Lambda(1)]_\omega)^{I(1)} = \bigoplus_{w\in W_0}C[\Lambda(1)]_\omega f_0T_{n_w}$.
Since $X_\emptyset = \bigoplus_{w\in W_0}C[\Lambda(1)]_\omega\otimes T_{n_w}$ (see after the proof of Lemma~\ref{lem:X_J in C}) and $X_\emptyset\to \pi_\emptyset$ sends $1\otimes 1$ to $f_0$ (see the proof of the previous lemma), we have $X_\emptyset\simeq \pi_\emptyset^{I(1)}$.
Note that $\supp f_0T_{n_w} = \overline{B}n_{w_\Delta w}I(1)$~\cite[IV.7. Proposition]{MR3600042}.
\end{rem}

\subsection{Filtrations}
In this subsection, we use the following notation: for $A\subset W_0$, $\overline{B}A\overline{B} = \bigcup_{v\in A}\overline{B}n_v\overline{B}$.

For a subset $A \subset W_0$ which is open (namely, if $v_1\in W_0$, $v_2\in A$ and $v_1\ge v_2$ then $v_1\in A$), we put
\[
\pi_{\emptyset,A} = \{f\in \Ind_{\overline{B}}^GC[\Lambda(1)]_\psi\mid \supp f\subset \overline{B}A\overline{B}\}.
\]
We also put 
\[
X_{\emptyset,A} = \bigoplus_{v\in A}n_{w_\Delta}C[\Lambda(1)]\otimes T_{n_{w_\Delta v}}.
\]
\begin{lem}
Let $h\in X_\emptyset$.
Then $h\in X_{\emptyset,A}$ if and only if its image in $\pi_{\emptyset}$ is in $\pi_{\emptyset,A}$.
Namely we have $X_{\emptyset,A} = X_\emptyset\cap \pi_A$.
\end{lem}
\begin{proof}
Let $H\in \pi_\emptyset$ be the image of $h$.
By the description of $X_\emptyset\to \pi_\emptyset$ (see Remark~\ref{rem:I-invariant and parabolic induction}), $h\in X_{\emptyset,A}$ if and only if $\supp H\subset \overline{B}AI(1)$.
For each $v\in A$, we have
\begin{align*}
\overline{B}vI(1) & = \overline{B}v(I(1)\cap v^{-1}\overline{B}v)(I(1)\cap v^{-1}Bv)\\
& = \overline{B}v(I(1)\cap v^{-1}Bv)\\
& \subset \overline{B}Bv\\
& \subset \bigcup_{v'\ge v}\overline{B}v'\overline{B}\\
& \subset \overline{B}A\overline{B}.
\end{align*}
Here we use \cite[Lemma~2.4]{MR2928148}.
Hence if $h\in X_{\emptyset,A}$ then $H\in \pi_{\emptyset,A}$.

Assume that $H\in \pi_{\emptyset,A}$ and $\supp(H)\cap \overline{B}vI(1) \ne\emptyset$ for $v\in W_0$.
Since $H$ is $I(1)$-invariant, we have $H(v)\ne 0$.
Therefore $v\in A$.
Hence $\supp(H)\subset \bigcup_{v\in A}\overline{B}vI(1)$.
We get $h\in X_{\emptyset,A}$.
\end{proof}

Set $X_{J,A} = X_J\cap X_{\emptyset,A}$ and $\pi_{J,A} = \pi_J\cap \pi_{\emptyset,A}$.
Let $w\in A$ be a minimal element and put $A' = A\setminus\{w\}$.
Then we have an embedding
\[
X_{\Delta,A}/X_{\Delta,A'}\hookrightarrow \pi_{\Delta,A}/\pi_{\Delta,A'}.
\]
For each $\alpha\in\Delta$, take a lift $a_\alpha\in \Lambda'_\alpha(1)$ of a generator of $\Lambda'_\alpha(1)/(Z_\kappa\cap \Lambda'_\alpha(1))$ such that $\langle \nu(a_\alpha),\alpha\rangle > 0$~\cite[III.4.]{MR3600042}.

The element $\#Z_\kappa^{-1}\sum_{\psi\in\omega}\sum_{t\in Z_\kappa}\psi(t)^{-1}\tau_{a_\alpha t}$ is in $C[\Lambda(1)]_\omega$ and does not depend on a choice of a lift (recall that $\psi$ is trivial on $Z_\kappa\cap \Lambda'_\alpha(1)$).
We denote it by $\tau_\alpha$.
Set $c_w = \prod_{w^{-1}(\alpha)>0}(1 - \tau_{\alpha})$.
Then as in \cite[V.8, Proposition]{MR3600042}, we have 
\begin{equation}\label{eq:G-rep, quotient of filtration}
\pi_{\Delta,A}/\pi_{\Delta,A'} = c_w(\pi_{\emptyset,A}/\pi_{\emptyset,A'}).
\end{equation}
We also have that $\pi_{\emptyset,A}/\pi_{\emptyset,A'}$ is free as $C[\Lambda(1)]_\omega$-module since it can be identified with the space of compactly supported functions on $\overline{B}\backslash \overline{B}w\overline{B}$ with values in $C[\Lambda(1)]_\omega$.
By the following lemma and \eqref{eq:G-rep, quotient of filtration}, $\pi_{\Delta,A}/\pi_{\Delta,A'}$ is also free.
\begin{lem}
The element $c_w\in C[\Lambda(1)]_\omega$ is not a zero divisor.
\end{lem}
\begin{proof}
The same proof in \cite[Lemma~3.10]{arXiv:1406.1003_accepted} can apply.
\end{proof}

\begin{lem}\label{lem:H-mod, quotient of filtration}
We have $X_{\Delta,A}/X_{\Delta,A'} = c_w(X_{\emptyset,A}/X_{\emptyset,A'})$.
\end{lem}
\begin{proof}
Since $X_{\Delta,A} = \pi_{\Delta,A}\cap X_{\emptyset,A}$, we have 
\[
X_{\Delta,A}/X_{\Delta,A'} = \pi_{\Delta,A}/\pi_{\Delta,A'}\cap X_{\emptyset,A}/X_{\emptyset,A'}
\]
and the right hand side is 
\[
c_w(\pi_{\emptyset,A}/\pi_{\emptyset,A'})\cap X_{\emptyset,A}/X_{\emptyset,A'}.
\]
Let $H$ be in this set.
Since $\pi_{\emptyset,A}/\pi_{\emptyset,A'}$ is a free $C[\Lambda(1)]_\omega$-module, the exact sequence $0\to \pi_{\emptyset,A'}\to \pi_{\emptyset,A}\to \pi_{\emptyset,A}/\pi_{\emptyset,A'}\to 0$ splits.
Hence $\pi_\emptyset\simeq \pi_{\emptyset,A'}\oplus(\pi_{\emptyset,A}/\pi_{\emptyset,A'})$.
Therefore $c_w\pi_{\emptyset,A}\simeq c_w\pi_{\emptyset,A'}\oplus c_w(\pi_{\emptyset,A}/\pi_{\emptyset,A'})$.
Hence $c_w(\pi_{\emptyset,A}/\pi_{\emptyset,A'}) \simeq (c_w\pi_{\emptyset,A})/(c_w\pi_{\emptyset,A'})$.
Hence there exists $H'\in \pi_{\emptyset,A}$ such that $H$ is the image of $c_wH'$.
Since $H = c_wH'\in X_{\Delta,A}/X_{\Delta,A'}$, there exists $h\in X_{\Delta,A}$ such that $c_wH' - h$ is zero in $X_{\Delta,A}/X_{\Delta,A'}$.
In particular it is zero in $c_w(\pi_{\emptyset,A}/\pi_{\emptyset,A'}) = (c_w\pi_{\emptyset,A})/(c_w\pi_{\emptyset,A'})$.
Therefore there exists $H''\in \pi_{\emptyset,A'}$ such that $c_wH' - h = c_wH''$.
Replacing $H'$ with $H' - H''$, we may assume $c_wH'\in X_{\emptyset,A}$.
Recall that $H'$ is a function with values in $C[\Lambda(1)]_\omega$.
Since the element $c_w$ is not a zero divisor in $C[\Lambda(1)]_\omega$, $c_wH'\in \pi_{\emptyset,A}$ implies $H'\in \pi_{\emptyset,A}$.
Since $c_wH'\in X_{\emptyset}$, $c_wH'$ is $I(1)$-invariant.
Hence $H'$ is also $I(1)$-invariant, again since $c_w$ is not a zero divisor.
Therefore $H'\in \pi_{\emptyset}^{I(1)} = X_\emptyset$.
Hence $H'\in X_\emptyset\cap \pi_{\emptyset,A} = X_{\emptyset,A}$.
Therefore $H\in c_w(X_{\emptyset,A}/X_{\emptyset,A'})$.
The reverse inclusion $c_w(\pi_{\emptyset,A}/\pi_{\emptyset,A'})\cap X_{\emptyset,A}/X_{\emptyset,A'}\supset  c_w(X_{\emptyset,A}/X_{\emptyset,A'})$ is obvious.
We get the lemma.
\end{proof}

\subsection{Proof of Lemma~\ref{lem:injective lemma, final form}}
Let $A,A',w$ be as in the previous subsection.
\begin{lem}\label{lem:splits, from the filtration}
The exact sequences of $C[\Lambda(1)]_\omega$-modules
\begin{gather*}
0\to \pi_{\Delta,A'}\to \pi_{\Delta,A}\to \pi_{\Delta,A}/\pi_{\Delta,A'}\to 0,\\
0\to X_{\Delta,A'}\to X_{\Delta,A}\to X_{\Delta,A}/X_{\Delta,A'}\to 0
\end{gather*}
split
\end{lem}
\begin{proof}
By \eqref{eq:G-rep, quotient of filtration} and from the fact that $\pi_{\emptyset,A}/\pi_{\emptyset,A'}$ is free, $\pi_{\Delta,A}/\pi_{\Delta,A'}$ is also free.
Hence the first exact sequence splits.
Using Lemma~\ref{lem:H-mod, quotient of filtration}, the same argument can apply for the second sequence.
\end{proof}
\begin{lem}\label{lem:X to pi has a section, on quotient}
The inclusion $X_{\Delta,A}/X_{\Delta,A'}\hookrightarrow \pi_{\Delta,A}/\pi_{\Delta,A'}$ has a section as $C[\Lambda(1)]_\omega$-modules.
\end{lem}
\begin{proof}
First we construct a section of $X_{\emptyset,A}/X_{\emptyset,A'}\to \pi_{\emptyset,A}/\pi_{\emptyset,A'}$.
Recall that $X_{\emptyset,A} = \pi_{\emptyset,A}^{I(1)}$.
Note that $X_{\emptyset,A}/X_{\emptyset,A'}\simeq C[\Lambda(1)]_\omega$ and the isomorphism is given by $f\mapsto f(w)$.
For $H\in \pi_{\emptyset,A}$, consider $H'\in \pi_{\emptyset,A}$ which is $I(1)$-invariant, $\supp(H') = \overline{B}vI(1)$ and $H'(v) = f(v)$.
Then $H\mapsto H'$ gives a section of $X_{\emptyset,A}/X_{\emptyset,A'}\to \pi_{\emptyset,A}/\pi_{\emptyset,A'}$.
Multiplying $c_w$ and using \eqref{eq:G-rep, quotient of filtration}, Lemma~\ref{lem:H-mod, quotient of filtration}, we get a section of $X_{\Delta,A}/X_{\Delta,A'}\to \pi_{\Delta,A}/\pi_{\Delta,A'}$.
\end{proof}

\begin{proof}[Proof of Lemma~\ref{lem:injective lemma, final form}]
Set $\pi^M_A = M\otimes_{C[\Lambda(1)]_\omega}\pi_{\Delta,A}$ and $X^M_A = M\otimes_{C[\Lambda(1)]_\omega}X_{\Delta,A}$.
Then by Lemma~\ref{lem:splits, from the filtration}, $\pi^M_{A'}$ (resp.\ $X^M_{A'}$) is a subspace of $\pi_A^M$ (resp.\ $X_A^M$).
By Lemma~\ref{lem:X to pi has a section, on quotient}, $X^M_A/X^M_{A'}\to \pi^M_A/\pi^M_{A'}$ is injective.

We prove that $X^M_{A}\to \pi^M_{A}$ is injective by induction on $\#A$.
We have the following diagram
\[
\begin{tikzcd}
0 \arrow[r] & X^M_{A'}\arrow[r]\arrow[d] & X^M_A\arrow[r]\arrow[d] & X^M_A/X^M_{A'}\arrow[r]\arrow[d] & 0\\
0 \arrow[r] & \pi^M_{A'}\arrow[r] & \pi^M_A\arrow[r] & \pi^M_A/\pi^M_{A'}\arrow[r] & 0.
\end{tikzcd}
\]
The homomorphism $X_{A'}^M\to \pi_{A'}^M$ is injective by inductive hypothesis and $X^M_A/X^M_{A'}\to \pi^M_A/\pi^M_{A'}$ is injective as we have seen.
Hence $X^M_A\to \pi^M_A$ is injective.
Setting $A = W_0$, we get the lemma.
\end{proof}

\section{Theorem}\label{sec:Theorem}
Let $\mathcal{C}_{\mathrm{f}}$ be the full-subcategory of $\mathcal{C}$ consisting of finite-dimensional modules.
Note that this category is closed under submodules, quotients and extensions.
\begin{thm}
Let $M\in \mathcal{C}_{\mathrm{f}}$.
Then $(X\otimes_{\mathcal{H}} \cInd_{I(1)}^G\trivrep)^{I(1)} \simeq X$.
\end{thm}
\begin{proof}
The theorem is true for simple $X$ by the main theorem of \cite{arXiv:1406.1003_accepted}, \cite[Theorem~4.17]{arXiv:1703.10384} and \cite[Theorem~5.11]{arXiv:1703.10384}.
We prove the theorem by induction on $\dim(M)$.

Assume that $M$ is not simple and let $M'$ be a proper nonzero submodule of $M$.
Let $\pi = \Ker(M'\otimes_{\mathcal{H}}\cInd_{I(1)}^G\trivrep\to M\otimes_{\mathcal{H}}\cInd_{I(1)}^G\trivrep)$.
By Theorem~\ref{thm:injective}, $M\to (M\otimes_{\mathcal{H}}\cInd_{I(1)}^G\trivrep)^{I(1)}$ is injective.
Then we have
\[
\begin{tikzcd}
0\arrow[r] & \pi^{I(1)} \arrow[r] & (M'\otimes_{\mathcal{H}}\cInd_{I(1)}^G\trivrep)^{I(1)} \arrow[r] & (M\otimes_{\mathcal{H}}\cInd_{I(1)}^G\trivrep)^{I(1)}\\
&& M'\arrow[u,"\wr"]\arrow[r,hookrightarrow] & M\arrow[u,hookrightarrow].
\end{tikzcd}
\]
Hence $\pi^{I(1)} = 0$.
Since $I(1)$ is a pro-$p$ group, $\pi = 0$.
Hence $M'\otimes_{\mathcal{H}}\cInd_{I(1)}^G\trivrep\to M\otimes_{\mathcal{H}}\cInd_{I(1)}^G\trivrep$ is injective.
Set $M'' = M/M'$.
Then we have a commutative diagram
\[
\begin{tikzcd}
0 \arrow[d] &0 \arrow[d] \\
(M' \otimes_\mathcal{H}\cInd_{I(1)}^G\trivrep)^{I(1)}\arrow[d] & M'  \arrow[d]\arrow[l,"\sim"'] \\
(M  \otimes_\mathcal{H}\cInd_{I(1)}^G\trivrep)^{I(1)}\arrow[d] & M   \arrow[d]\arrow[l]       \\
(M''\otimes_\mathcal{H}\cInd_{I(1)}^G\trivrep)^{I(1)}          & M'' \arrow[d]\arrow[l,"\sim"'] \\
& 0.
\end{tikzcd}
\]
with exact columns.
Therefore $M\to (M\otimes_\mathcal{H}\cInd_{I(1)}^G\trivrep)^{I(1)}$ is isomorphic.
\end{proof}
\begin{cor}
Let $\mathcal{C}_{G,\mathrm{f}}$ be the category of representations of $G$ consisting of the following objects:
\begin{itemize}
\item has a finite length.
\item any irreducible subquotient is a subquotient of $\Ind_B^G\sigma$ for a irreducible representation $\sigma$ of $Z$.
\item is generated by $I(1)$-invariants.
\end{itemize}
Then $\mathcal{C}_{\mathrm{f}}\simeq \mathcal{C}_{G,\mathrm{f}}$.
The equivalence is given by $\pi\to \pi^{I(1)}$ and $M\mapsto M\otimes_{\mathcal{H}}\cInd_{I(1)}^GM$.
\end{cor}
\begin{proof}
By the classification theorem in \cite{MR3600042} and a result in \cite[Theorem~5.11]{arXiv:1703.10384}, if $\pi\in \mathcal{C}_{G,\mathrm{f}}$ is irreducible, then $\pi^{I(1)}\in \mathcal{C}_{\mathrm{f}}$.
Hence, by induction on the length, if $\pi\in \mathcal{C}_{G,\mathrm{f}}$ then $\pi^{I(1)}\in \mathcal{C}_{\mathrm{f}}$.

Let $\pi \in \mathcal{C}_{G,\mathrm{f}}$ and we prove that $\pi^{I(1)}\otimes_{\mathcal{H}}\cInd_{I(1)}^G\trivrep\to \pi$ is an isomorphism.
The homomorphism is sujrective since $\pi$ is generated by $\pi^{I(1)}$.
Let $\pi'$ be the kernel.
Then we have an exact sequence
\[
0\to (\pi')^{I(1)}\to (\pi^{I(1)}\otimes_{\mathcal{H}}\cInd_{I(1)}^G\trivrep)^{I(1)}\to \pi^{I(1)}
\]
and the last map is isomorphism by the theorem.
Hence $(\pi')^{I(1)} = 0$ and it implies $\pi' = 0$.
Therefore the homomorphism is also injective.
Combining with the previous theorem, we have proved the desired equivalence of categories.
\end{proof}

\end{document}